\documentclass[11pt]{article}
\usepackage{amsmath,epsfig,subfigure,amssymb,color,latexsym,amsfonts,amsthm}
\usepackage{graphicx,graphics,epstopdf}

\begin{document}
\newtheorem{thm}{Theorem}[section]
\newtheorem{cor}{Corollary}[section]
\newtheorem{lem}{Lemma}[section]
\newtheorem{prop}{Proposition}[section]
\newtheorem{defn}{Definition}[section]
\newtheorem{rk}{Remark}[section]
\newtheorem{nota}{Notation}[section]
\newtheorem{Ex}{Exmple}[section]
\def\nm{\noalign{\medskip}}

\numberwithin{equation}{section}

\newcommand{\ds}{\displaystyle}
\newcommand{\pf}{\medskip \noindent {\sl Proof}. ~ }
\newcommand{\p}{\partial}
\renewcommand{\a}{\alpha}
\newcommand{\z}{\zeta}
\newcommand{\pd}[2]{\frac {\p #1}{\p #2}}
\newcommand{\norm}[1]{\left\| #1 \right \|}
\newcommand{\dbar}{\overline \p}
\newcommand{\eqnref}[1]{(\ref {#1})}
\newcommand{\na}{\nabla}
\newcommand{\Om}{\Omega}
\newcommand{\ep}{\epsilon}
\newcommand{\tmu}{\widetilde \epsilon}
\newcommand{\vep}{\varepsilon}
\newcommand{\tlambda}{\widetilde \lambda}
\newcommand{\tnu}{\widetilde \nu}
\newcommand{\vp}{\varphi}
\newcommand{\RR}{\mathbb{R}}
\newcommand{\CC}{\mathbb{C}}
\newcommand{\NN}{\mathbb{N}}
\renewcommand{\div}{\mbox{div}~}
\newcommand{\bu}{{\bf u}}
\newcommand{\la}{\langle}
\newcommand{\ra}{\rangle}
\newcommand{\Scal}{\mathcal{S}}
\newcommand{\Lcal}{\mathcal{L}}
\newcommand{\Kcal}{\mathcal{K}}
\newcommand{\Dcal}{\mathcal{D}}
\newcommand{\tScal}{\widetilde{\mathcal{S}}}
\newcommand{\tKcal}{\widetilde{\mathcal{K}}}
\newcommand{\Pcal}{\mathcal{P}}
\newcommand{\Qcal}{\mathcal{Q}}
\newcommand{\id}{\mbox{Id}}
\newcommand{\stint}{\int_{-T}^T{\int_0^1}}

\newcommand{\be}{\begin{equation}}
\newcommand{\ee}{\end{equation}}

\newcommand{\rd}{{\mathbb R^d}}
\newcommand{\rr}{{\mathbb R}}
\newcommand{\alert}[1]{\fbox{#1}}
\newcommand{\eqd}{\sim}
\def\R{{\mathbb R}}
\def\N{{\mathbb N}}
\def\Q{{\mathbb Q}}
\def\C{{\mathbb C}}
\def\ZZ{{\mathbb Z}}
\def\l{{\langle}}
\def\r{\rangle}
\def\t{\tau}
\def\k{\kappa}
\def\a{\alpha}
\def\la{\lambda}
\def\De{\Delta}
\def\de{\delta}
\def\ga{\gamma}
\def\Ga{\Gamma}
\def\ep{\varepsilon}
\def\eps{\varepsilon}
\def\si{\sigma}
\def\Re {{\rm Re}\,}
\def\Im {{\rm Im}\,}
\def\E{{\mathbb E}}
\def\P{{\mathbb P}}
\def\Z{{\mathbb Z}}
\def\D{{\mathbb D}}
\def\p{\partial}
\newcommand{\ceil}[1]{\lceil{#1}\rceil}

\title{Numerical study of vanishing and spreading dynamics of chemotaxis systems with logistic source and a free boundary}

\author{Lei Yang \thanks{School of Mathematics, Jilin University, Changchun, 130012, P. R. China.} \,\, and
Lianzhang Bao\thanks{School of Mathematics, Jilin University, Changchun, 130012, P. R. China, and Department of Mathematics and Statistics,
Auburn University,  AL 36849, U. S. A. (lzbao@jlu.edu.cn).}}

\date{}

\maketitle

\begin{abstract}
The current paper is to investigate the numerical approximation of logistic type chemotaxis models in one space dimension with a free boundary.
Such a model with a free boundary describes the spreading of a new or invasive species subject to the influence of some chemical substances
in an environment with a free boundary representing the spreading front (see Bao and Shen \cite{BaoShen1}, \cite{BaoShen2}). The main challenges in the numerical studies lie in tracking the moving free boundary and the nonlinear terms from chemical. To overcome them, a front fixing framework coupled with finite difference method is introduced. The accuracy of the proposed method, the positivity of the solution, and the stability of the scheme are discussed. The numerical simulations agree well with theoretical results such as the vanishing spreading dichotomy, local persistence, and stability. These simulations also validate some conjectures in our future theoretical studies such as the dependence of the vanishing-spreading dichotomy on the initial solution $u_0$, initial habitat $h_0$, the moving speed $\nu$ and the chemotactic sensitivity coefficients $\chi_1,\chi_2$.
\end{abstract}


\medskip


\section{Introduction}
The current paper is  to  study, in particular, numerically, the spreading and vanishing dynamiccs of
 the following attraction-repulsion chemotaxis system with a  free boundary and logistic source,
\begin{equation}\label{one-free-boundary-eq}
\begin{cases}
u_t = u_{xx} -\chi_1  (u  v_{1,x})_x + \chi_2 (u v_{2,x})_x + u(a(t,x) - b(t,x)u), \quad 0<x<h(t)
\\
 0 = \partial_{xx} v_1 - \lambda _1v_1 + \mu_1u,  \quad  0<x<h(t)
 \\
 0 = \partial_{xx}v_2 - \lambda_2 v_2 + \mu_2u,  \quad 0<x<h(t)
 \\
 h'(t) = -\nu u_x(t,h(t))
\\
u_x(t,0) = v_{1,x}(t,0) = v_{2,x}(t,0) = 0
\\
 u(t,h(t)) = v_{1,x}(t,h(t)) = v_{2,x}(t,h(t)) = 0
 \\
 h(0) = h_0,\quad u(x,0) = u_0(x),\quad  0\leq x\leq h_0,
\end{cases}
\end{equation}
where $\nu>0$ in \eqref{one-free-boundary-eq} is a positive constant, $a, b ,\chi_i$, $\lambda_i$, and $\mu_i$ ($i=1,2$)  are nonnegative constants, and $a(t,x)$ and $b(t,x)$  satisfy the following assumption,

\medskip

\noindent {\bf (H0)} {\it $a(t,x)$ and $b(t,x)$ are bounded $C^1$ functions on $\RR\times [0,\infty)$,
and
$$a_{\inf}:=\inf_{t\in\RR,x\in [0,\infty)}a(t,x)>0,\quad b_{\inf}:=\inf_{t\in\RR,x\in[0,\infty)}b(t,x)>0.
$$
}
Biological backgrounds of \eqref{one-free-boundary-eq} are discussed in the paper (\cite{BaoShen1}). The free boundary condition in \eqref{one-free-boundary-eq} is also derived in \cite{BaoShen1} based on the consideration of ``population loss'' at the front which assumes that the expansion of the spreading front is evolved in a way that the average population density loss near the front is kept at a certain preferred level of the species, and for each given species in a given homogeneous environment, this preferred density level is a positive constant determined by their specific social and biological needs, and the environment.

One of the first mathematical models of chemotaxis were introduced by Keller and Segel (\cite{Keller1970Initiation}, \cite{Keller1971Model}) to describe the aggregation of certain type of bacteria in 1970. Since their publications, considerable progress has been made in the analysis of various particular case of chemotaxis (Keller-Segel) model on both bounded and unbounded fixed domain (see \cite{Bellomo2015Toward}, \cite{Diaz1995Symmtr}, \cite{Diaz1998Symmtr}, \cite{Galakhov2016on}, \cite{Horstmann2005bound}, \cite{Kanga2016blow}, \cite{Nagai1997application}, \cite{Sugiyama2006global1}, \cite{Sugiyama2006global2}, \cite{Wang2014on}, \cite{Winkler2010aggregation}, \cite{Winkler2011blow}, \cite{Winkler2013finite}, \cite{Winkler2014global}, \cite{Winkler2014how}, \cite{Yokota2015existence}, \cite{Zheng2015boundedness}, and the references therein).
Among the fundamental problems in studying chemotaxis model are the existence of nonnegative solutions which are globally defined in time or blow up at a finite time and the asymptotic behavior of time global solutions.

Du and Lin studied the population invasion represented by \emph{Fisher-KPP} free boundary problem in 2010 \cite{DuLi}. The breaking difference between the asymptotic behaviours of \emph{Fisher-KPP} with a free boundary and on the fixed or fixed unbounded domain is the vanishing-spreading dichotomy, which is well supported by some empirical evidences, for example, the introduction of several bird species from Europe to North America in the 1900s was successful only after many initial attempts (see \cite{lockwood2007invasion},\cite{Shigesada1997biological}).

Compared to the studying chemotaxis model on fixed bounded or fixed unbounded domain and the asymptotic behaviour of \emph{Fisher-KPP} equation with a free boundary, the
central problems in studying system \eqref{one-free-boundary-eq} are the existence of nonnegative solutions which are globally defined
in time, the vanishing-spreading dichotomy, local persistence, local stability, and so on.
%
%

To state the main results of the current paper, we first recall some theoretical results proved in \cite{BaoShen2} which will all be validated in our numerical simulations.

Let
$$
C_{\rm unif}^b(\RR^+)=\{u\in C(\RR^+)\,|\, u(x)\,\, \text{is uniformly continuous and bounded on}\,\, \RR^+\}
$$
with norm $\|u\|_\infty=\sup_{x\in\RR^+}|u(x)|$, and
$$
C_{\rm unif}^b(\RR)=\{u\in C(\RR)\,|\, u(x)\,\, \text{is uniformly continuous and bounded on}\,\, \RR\}
$$
with norm $\|u\|_\infty=\sup_{x\in\RR}|u(x)|$. Define
\begin{align}\label{m-eq}
M= \min\Big\{ &\frac{1}{\lambda_2}\big( (\chi_2\mu_2\lambda_2-\chi_1\mu_1\lambda_1)_+ + \chi_1\mu_1(\lambda_1-\lambda_2)_+ \big),\nonumber\\
& \qquad \frac{1}{\lambda_1}\big( (\chi_2\mu_2\lambda_2-\chi_1\mu_1\lambda_1)_+ + \chi_2\mu_2(\lambda_1-\lambda_2)_+ \big) \Big\}
 \end{align}
 and
 \begin{align}\label{k-eq}
K=\min\Big\{&\frac{1}{\lambda_2}\Big(|\chi_1\mu_1\lambda_1-\chi_2\mu_2\lambda_2|+\chi_1\mu_1|\lambda_1-\lambda_2|\Big),\nonumber\\
&\quad  \frac{1}{\lambda_1}\Big(|\chi_1\mu_1\lambda_1-\chi_2\mu_2\lambda_2|+\chi_2\mu_2|\lambda_1-\lambda_2|\Big) \Big\}.
\end{align}
  Let {\bf (H1)}- {\bf (H3)}  be the following standing assumptions.

\medskip

\noindent {\bf (H1)}  $b_{\inf}>\chi_1\mu_1-\chi_2\mu_2+M$.

\medskip

\noindent {\bf (H2)} $b_{\inf}>\Big(1+\frac{a_{\sup}}{a_{\inf}}\Big)\chi_1\mu_1-\chi_2\mu_2+M$.

\medskip

\noindent {\bf (H3)}  $b_{\inf}>\chi_1\mu_1-\chi_2\mu_2+K$.

\medskip

Note  that
$$
M\le \chi_2\mu_2.
$$
Hence $b_{\inf}\ge \chi_1\mu_1$ implies {\bf (H1)}. In the case $\chi_2=0$, we can choose $\lambda_2=\lambda_1$,  and then
$M=0$ and $K=\chi_1\mu_1$. Hence {\bf (H1)} becomes
$b_{\inf}>\chi_1\mu_1$, {\bf (H2)} becomes $b_{\inf}>(1+\frac{a_{\sup}}{a_{\inf}})\chi_1\mu_1$, and {\bf (H3)} becomes $b_{\inf}>2\chi_1\mu_1$. In the case $\chi_1=0$, we can also choose $\lambda_1=\lambda_2$, and then $M=\chi_2\mu_2$ and $K=\chi_2\mu_2$. Hence {\bf (H1)}
(resp.{\bf (H2)}, {\bf (H3)}) becomes $b_{\inf}>0$.
Biologically, {\bf (H1), (H2)}, and {\bf (H3)} indicate that the chemo-attraction sensitivity is relatively small with respect to logistic damping.

\smallskip

When {\bf (H1)} holds, we put
\begin{equation}
\label{M0-eq}
M_0=\frac{a_{\sup}}{b_{\inf}+\chi_2\mu_2-\chi_1\mu_1-M}
\end{equation}
and
\begin{equation}
\label{m0-eq}
m_0= \frac{a_{\inf}\big(b_{\inf}-(1+\frac{a_{\sup}}{a_{\inf}})\chi_1\mu_1+\chi_2\mu_2-M\big)}{(b_{\inf}-\chi_1\mu_1+\chi_2\mu_2-M)
(b_{\sup}-\chi_1\mu_1+\chi_2\mu_2)}.
\end{equation}
Note that if {\bf (H2)} holds, then $m_0>0$.

Let
$$
H(a,b)={\rm cl}\{ (a(t+\cdot,\cdot),b(t+\cdot,\cdot))|t\in\RR\}
$$
with open compact topology, where the closure is taken under the open compact topology.

The main results of the paper \cite{BaoShen2} are stated in the following.
The first result is on the global existence of nonnegative solutions of \eqref{one-free-boundary-eq}.

 \emph{Global existence \cite[Theorem 1.2]{BaoShen2}:
 \label{free-boundary-thm1}
  If {\bf (H1)} holds, then for any $t_0\in\RR$, and any $h_0>0$ and any function $u_0(x)$ on $[0,h_0]$ satisfying
  \begin{equation}
  \label{initial-cond-eq}
  u_0\in C^2[0,h_0],\quad u_0(x)\ge 0\,\, {\rm for}\,\, x\in [0,h_0],\quad {\rm and}\,\,  u_0'(0)=0,\,\,  u_0(h_0)=0,
  \end{equation}
   \eqref{one-free-boundary-eq} has a unique globally defined solution $(u(t,x;t_0,u_0,h_0)$, $v_1(t,x;t_0,u_0,h_0)$, $v_2(t,x;t_0,u_0,h_0)$, $h(t;t_0,u_0,h_0))$
with $u(t_0,x;t_0,u_0,h_0)=u_0(x)$ and $h(t_0;t_0,u_0,h_0)=h_0$. Moreover,
\begin{equation}\label{thm1-h-bound}
{ 0\leq h'(t) \leq 2\nu M_1 C_0},
\end{equation}
 \begin{equation}
 \label{thm1-eq1}
0\le u(t,x;t_0,u_0,h_0)\le  \max\{\|u_0\|_\infty, M_0\}\quad \forall \,\, t\in [t_0,\infty), \,\, x\in [0,h(t;t_0,u_0,h_0))
\end{equation}
and
\begin{equation}
\label{thm1-eq2}
\limsup_{t\to\infty} \sup_{x\in [0,h(t;t_0,u_0,h_0))}u(t_0+t,x;t_0,u_0,h_0)\le M_0,
\end{equation}
where $M_1$ is a big enough constant and $C_0 = \max\{\|u_0\|_{\infty}, M_0\}$.
}
\smallskip

Assume {\bf (H1)}. For any given $t_0\in\RR$,  and any given $h_0>0$ and  $u_0(\cdot)$ satisfying
\eqref{initial-cond-eq}, by the nonnegativity of $u(t,x;t_0,u_0,h_0)$, $h^{'}(t;t_0,u_0,h_0)\ge 0$ for all $t>t_0$. Hence
$\lim_{t\to\infty}h(t;t_0,u_0,h_0)$ exists. Put
$$
h_\infty(t_0,u_0,h_0)=\lim_{t\to\infty} h(t;t_0,u_0,h_0).
$$
We say {\it vanishing} occurs if $h_\infty(t_0,u_0,h_0)<\infty$ and
$$
\lim_{t\to\infty}\|u(t,\cdot;t_0,u_0,h_0)\|_{C([0,h(t;t_0,u_0,h_0)])}=0.
$$
We say {\it spreading} occurs if $h_\infty(t_0,u_0,h_0)=\infty$ and
for any $L>0$, $$\liminf_{t\to\infty} \inf_{0\le x \le L} u(t,x;u_0,h_0)>0.$$

For given $l>0$,  consider the following linear equation,
\begin{equation}\label{Eq-PLE1}
\begin{cases}
 v_t = v_{xx} + a(t,x) v, \quad 0<x<l\cr
 v_x(t,0) = v(t,l) =0.
\end{cases}
\end{equation}
Let $[\lambda_{\min}(a,l),\lambda_{\max}(a,l)]$ be the principal spectrum interval of \eqref{Eq-PLE1}
(see Definition 2.1 \cite{BaoShen2}). Let $l^*>0$ be such that
$\lambda_{\min}(a,l)>0$ for $ l>l^*$ and $\lambda_{\min}(a,l^*)=0$ (see Lemma 2.2, 2.3 \cite{BaoShen2} for the existence and uniqueness
of $l^*$).

The second result is about the spreading and vanishing dichotomy scenario in \eqref{one-free-boundary-eq}.

\emph{Spreading-vanishing dichotomy \cite[Theorem 1.3]{BaoShen2}:
Assume that {\bf (H1)} holds.
 For any given $t_0\in\RR$, and $h_0>0$ and $u_0(\cdot)$ satisfying \eqref{initial-cond-eq}, we have that
either
(i) vanishing occurs and $h_\infty(t_0,u_0,h_0)\le l^*$; or
(ii) spreading occurs.}

\smallskip
For given $t_0\in\RR$,  and  $h_0>0$ and   $u_0(\cdot)$ satisfying
\eqref{initial-cond-eq}, if spreading occurs, it is interesting to know whether {\it local uniform persistence} occurs in the sense that
there is a positive constant $\tilde m_0$ independent of the initial data such that for any $L>0$,
$$\liminf_{t\to\infty}\inf_{0\le x\le L}u(t,x;t_0,u_0,h_0)\ge \tilde m_0,
$$
   and whether {\it local uniform convergence} occurs in the sense that  $\lim_{t\to\infty} u(t,x;t_0,u_0,h_0)$ exists locally uniformly. We have the following result
along this direction.

\emph{Persistence and convergence \cite[Theorem 1.4]{BaoShen2}:}
\label{free-boundary-thm3}
Assume that {\bf (H1)} holds and that $h_0>0$ and $u_0(\cdot)$ satisfy \eqref{initial-cond-eq}.

\smallskip

\noindent \emph{(i) (Local uniform persistence)} For any given $t_0\in\RR$, if  $h_\infty(t_0,u_0,h_0)=\infty$ and {\bf (H2)} holds, then
for any $L>0$, $$\liminf_{t\to\infty} \inf_{0\le x \le L} u(t,x;t_0,u_0,h_0)>m_0,$$
where $m_0$ is as in \eqref{m0-eq}.

\smallskip
\noindent \emph{(ii) (Local uniform convergence)}  Assume that {\bf (H3)} holds,
and  that for any $(\tilde a,\tilde b)\in H(a,b)$, there
has a  unique strictly positive entire solution $(u^*(t,x;\tilde a,\tilde b),v_1^*(t,x;\tilde a,\tilde b)$, $v_2^*(t,x;\tilde a,\tilde b))$.
Then for  any given $t_0\in\RR$, if
  $h_\infty(t_0,u_0,h_0)=\infty$, there are $\chi_1^* > 0, \chi_2^* > 0$ such that to any {$0\leq \chi_1\leq \chi_1^*$, $0\leq \chi_2\leq \chi_2^*$,}
for any $L>0$,
\begin{equation}\label{FBP-stability}
\lim_{t\to\infty} \sup_{0\le x\le L} |u(t,x;t_0,u_0,h_0)-u^*(t,x;a,b)|=0.
\end{equation}
\smallskip

\noindent \emph{(iii) (Local uniform convergence)}
Assume that {\bf (H3)} holds, and that $a(t,x)\equiv a(t)$ and $b(t,x)\equiv b(t)$. Then for any given  $t_0\in\RR$, if $h_\infty(t_0,u_0,h_0)=\infty$, then
for any $L>0$,
$$
\lim_{t\to\infty} \sup_{0\le x\le L} |u(t,x;u_0,h_0)-u^*(t)|=0,
$$
where $u^*(t)$ is the unique strictly  positive entire solution of the ODE
\begin{equation}
\label{fisher-kpp-ode}
u^{'}=u(a(t)-b(t)u)
\end{equation}
(see \cite[Lemma 2.5]{Issa-Shen-2017} for the existence and uniqueness of strictly  positive entire solutions of
\eqref{fisher-kpp-ode}).

\begin{rk}
Biologically, the invasion or spreading of the population is depending on the initial solution, initial habitat, the moving speed $\nu$ (\cite{DuLi}). When the spreading happens, the local persistence and convergence can be guaranteed in \emph{Fisher-KPP} equation with a free boundary, furthermore, there is a asymptotic spreading speed such that $\lim_{t\to\infty}\frac{h(t)}{t} = c^* > 0$ (see \cite{DuLi}, \cite{Li2016duffusive}).
\end{rk}
Compared to the vanishing-spreading dichotomy in \emph{Fisher-KPP} equation with a free boundary, the chemotaxis system \eqref{one-free-boundary-eq} do not have comparison principle which leads the following interesting  open problems but has positive answers in \emph{Fisher-KPP} case.
\begin{enumerate}
\item For given $u_0(\cdot)$ and $0<h_0<l^*$, whether there is
$\nu^*>0$ such that for $0<\nu\le \nu^*$, vanishing occurs, and for $\nu>\nu^*$, spreading occurs.
\item For given $\nu>0$, $\phi(\cdot)$, and
$0<h_0<l^*$, whether there is $\sigma^*>0$ such that for $u_0=\sigma \phi$ with $\sigma\le \sigma^*$, vanishing occurs, and
for $u_0=\sigma\phi$ with $\sigma>\sigma^*$, spreading occurs.
\item Whether there is a spreading speed $c^*>0$ such that
$\lim_{t\to\infty}\frac{h(t)}{t}=c^*$ as long as the spreading occurs.
\end{enumerate}

\begin{rk}
Vanishing-spreading result \cite[Theorem 1.2]{BaoShen2} indicate there is a separating value $l^*$, which is independent of the chemotactic sensitivity coefficients $\chi_1, \chi_2$, such that in the vanishing scenario the limiting moving boundary $h_\infty <l^*$ and when the initial habitat $h_0 > l^*$ the spreading guaranteed. The dependence of the dynamics of the system on the chemotactic sensitivity coefficients is another important and interesting questions \cite{SaShXu}, \cite{Winkler2014how}. We also have the following question in this direction.
\begin{description}
\item[4.] If the asymptotic spreading speed exist, whether the limit $\lim_{t\to\infty}\frac{h(t)}{t}=c^*$ depends on the chemotactic sensitivity coefficient $\chi_1$ and $\chi_2$.
\end{description}
\end{rk}

The objective of the current paper is to study the numerical effect of the parameters $\sigma, \nu, \chi_1, \chi_2$ on the vanishing and spreading dynamics in the system \eqref{one-free-boundary-eq} which will give us directions in the theoretical studies. For simplicity, we only consider the constant logistic coefficients in the system, where $a(t,x) = a, b(t,x) =b$. In general, it is always difficult to handle the attraction term in chemotaxis system which may lead to convection dominant in the system \cite{Chiu2007optimal}, \cite{Li2017local}, \cite{Liu2018positivity}, \cite{Saito2005Notes}. However, in the system \eqref{one-free-boundary-eq}, we have extra numerical challenges in efficiently and accurately handling the moving boundaries \cite{Piqueras2017afront}. These two challenges require us to construct a new numerical algorithm in the numerical study.

Thanks to the maximum principle in the elliptic equation and death damping coefficient $b$ in the parabolic equation, the chemoattraction term can be controlled by the magnitude of the population density which has a global bound by the suppressing of the death rate.
The front-fixing method has been successfully applied to solve one dimensional free boundary problem \cite{Landau1950heat}, \cite{Liu2018numerical}, \cite{Liu2019numerical}, \cite{Piqueras2017afront} which changes the moving boundary to fixed domain, and is the main concern in our numerical studies of \eqref{one-free-boundary-eq}. Combined with the front-fixing method, finite difference in parabolic and finite volume method in the elliptic equations in the system \eqref{one-free-boundary-eq}, we construct a new algorithm in the numerical study, as a by-product we also obtain the consistency, monotonicity of the moving boundary, positivity of the solution and stability results.

Our numerical experiments validate the vanishing and spreading dichotomy in the numerical scheme of system  \eqref{one-free-boundary-eq} which is similar to \emph{Fisher-KPP} equation with a free boundary and give evidences to our conjectures that:
\begin{enumerate}
\item For given $u_0(\cdot)$ and $0<h_0<l^*$, there is
$\nu^*>0$ such that for $0<\nu\le \nu^*$, vanishing occurs, and for $\nu>\nu^*$, spreading occurs. Which means in order to spread to the half space $\mathbb{R}^+$, the moving speed $\nu$ should be large enough and otherwise the population will be extinct.
\item For given $\nu>0$, $\phi(\cdot)$, and
$0<h_0<l^*$, there is $\sigma^*>0$ such that for $u_0=\sigma \phi$ with $\sigma\le \sigma^*$, vanishing occurs, and
for $u_0=\sigma\phi$ with $\sigma>\sigma^*$, spreading occurs. Biologically, large initial population density helps the establishment and spreading of invasion which is an indirect evidence to the early birds introduction problem in 1900s \cite{lockwood2007invasion}, \cite{Shigesada1997biological}.
\item There is a spreading speed $c^*>0$ such that
$\lim_{t\to\infty}\frac{h(t)}{t}=c^*$ as long as the spreading occurs, which is independent of chemotactic sensitivity coefficient $\chi_1$ and $\chi_2$. Chemical $v_1$ and $v_2$ are produced by the species and the density is close to zero near the spreading front. In such case the decisive effect of the spreading speed should not depend on the chemotactic sensitivity coefficients $\chi_1$ and $\chi_2$.
\end{enumerate}

\smallskip

The rest of this paper is organized in the following way. In section 2, we first use \emph{Landau} transformation to transfer the moving boundary to a fixed domain, then we use the finite difference, finite volume, and iteration method to approximate the continuous chemotaxis system. We also prove the monotonicity of the moving boundary, the positivity and stability of the discrete solutions. In section 3, we study the numerical spreading-vanishing dichotomy in  \eqref{one-free-boundary-eq} which validates our theoretical results (Vanishing-spreading dichotomy, local persistence and convergence). Our simulations also indicate the dependence or independence of the vanishing-spreading dichotomy on parameters $\nu, u_0, h_0, \chi_1, \chi_2$ and so on. In section 4, some future works are briefly discussed.

\section{Numerical approximation of the free boundary problem}
In this section, we study the numerical approximation of system \eqref{one-free-boundary-eq} with constant logistic coefficients $a(t,x) =a, b(t,x)=b$. First through the well-known \emph{Landau} transformation (see \cite{Landau1950heat}), we convert \eqref{one-free-boundary-eq} into a fixed spatial domain problem. In such a way the length of the moving boundary is included as another variable to be solved apart from the population density. Then we solve the converted new problem on the basis of finite difference and finite volume method. There is a circulation that each time when time variable increase we solve the elliptic equations first and using the forward differential method to find the solution of the parabolic equation for the new time.

From the elliptic equations in \eqref{one-free-boundary-eq}, we know that
\begin{eqnarray}
\partial_{xx}v_{1}&=&\lambda_{1}v_{1}-\mu_{1}u,\label{one-elliptic-eq1}
\\
\partial_{xx}v_{2}&=&\lambda_{2}v_{2}-\mu_{2}u. \label{one-elliptic-eq2}
\end{eqnarray}
Combining \eqref{one-elliptic-eq1}, \eqref{one-elliptic-eq2} and the first equation in \eqref{one-free-boundary-eq}, we have
\begin{equation}\label{one-free-boundary-eq1}
u_{t}=u_{xx}+(-\chi_{1}v_{1x}+\chi_{2}v_{2x})u_{x}+(-\chi_{1}\lambda_{1}v_{1}+\chi_{2}\lambda_{2}v_{2}+a)u+(\chi_{1}\mu_{1}-\chi_{2}\mu_{2}-b)u^{2}.
\end{equation}
Then we introduce the \emph{Landau} transformation,
\begin{equation}
z(t,x)=\frac{x}{h(t)},\quad w(t,z)=u(t,x),\quad V_1(t,z)=v_1(t,x),\quad V_2(t,z)=v_2(t,x).
\end{equation}
Under this substitution the elliptic equations \eqref{one-elliptic-eq1}, \eqref{one-elliptic-eq2} take the form:
\begin{eqnarray}
\frac{\partial^{2}V_{1}}{\partial z^{2}}\cdot\frac{1}{h^{2}(t)}-\lambda_{1}V_{1}+\mu_{1}w=0,\qquad 0<z<1,\label{elliptic-eq1}
\\
\frac{\partial^{2}V_{2}}{\partial z^{2}}\cdot\frac{1}{h^{2}(t)}-\lambda_{2}V_{2}+\mu_{2}w=0,\qquad 0<z<1.\label{elliptic-eq2}
\end{eqnarray}
The elliptic boundary conditions are
\begin{eqnarray}
V_{1,z}(t,0)=V_{2,z}(t,0)=0, \label{elliptic-condition1}
\\
V_{1,z}(t,1)=V_{2,z}(t,1)=0. \label{elliptic-condition2}
\end{eqnarray}
Equation \eqref{one-free-boundary-eq1} takes the form:
\begin{equation}\label{one-free-boundary-eq2}
\begin{split}
\frac{\partial w}{\partial t}+\frac{\partial w}{\partial z}(-\frac{h^{'}(t)}{h(t)}z)=&\frac{\partial^{2}w}{\partial z^{2}}\frac{1}{h^{2}(t)}+(-\chi_{1}V_{1z}+\chi_{2}V_{2z})\frac{\partial w}{\partial z}\frac{1}{h^2(t)}+(-\chi_{1}\lambda_{1}V_{1}\\& +\chi_{2}\lambda_{2}V_{2}+a)w+(\chi_{1}\mu_{1}-\chi_{2}\mu_{2}-b)w^{2},\qquad 0<z<1.
\end{split}
\end{equation}
Let $G(t)$ denote $h^{2}(t)$ and multiply it on both sides of the above equation.
\begin{equation}\label{one-free-boundary-eq3}
\begin{split}
G(t)\frac{\partial w}{\partial z}+\frac{\partial w}{\partial z}(-G^{'}(t)\cdot\frac{z}{2})=&\frac{\partial^{2}w}{\partial z^{2}}+(-\chi_{1}V_{1z}+\chi_{2}V_{2z})\frac{\partial w}{\partial z}+(-\chi_{1}\lambda_{1}V_{1} \\&+\chi_{2}\lambda_{2}V_{2}+a)w\cdot G(t)\\&+(\chi_{1}\mu_{1}-\chi_{2}\mu_{2}-b)w^{2}\cdot G(t),\qquad 0<z<1.
\end{split}
\end{equation}
Boundary conditions and \emph{Stefan} condition take the form
\begin{equation}\label{boundary-cond1}
\frac{\partial w}{\partial z}(t,0)=0,\quad w(t,1)=0,\quad t>0
\end{equation}
and
\begin{equation}\label{boundary-cond2}
G^{'}(t)=-2\nu\frac{\partial w}{\partial z}(t,1),\quad t>0.
\end{equation}
The initial conditions in \eqref{one-free-boundary-eq} become:
\begin{equation}\label{initial-cond1}
G(0)=h_{0}^{2},\quad w(0,z)=w_{0}(z)=U_{0}(z\cdot h_0),\quad 0\leq z\leq1,
\end{equation}
and the initial function $u_{0}(x)$ is changed into $w_{0}(z)$ which maintains:
\begin{equation}\label{initial-cond2}
w^{'}_{0}(0)=w_{0}(1)=0,\qquad w_{0}(z)>0,\quad 0\leq z<1.
\end{equation}
Under the transformation, our aim is to solve the nonlinear parabolic partial differential system \eqref{elliptic-eq1}, \eqref{elliptic-eq2}, \eqref{one-free-boundary-eq3} in the fixed domain $(0,\infty)\times(0,1)$ for the variables $(t,z)$.

The following is the process according to the theory of the finite difference method. First we consider the  time and space discretization $\tau=\triangle t$, $h=\triangle z=1/M$, which means the interval (0,1) is divided into $M$ equal cells
\begin{displaymath}
0=z_0<z_1<\cdots<z_M=1,
\end{displaymath}
and the mesh points $(t^{n},z_{j})$, with $t^{n}=n\tau,n\geq0,z_{j}=jh,0\leq j\leq M$. For abbreviation, the approximate value of $w(t^n,z_j)$ can be denoted by $w^{n}_{j}$, the approximate values of $V_{1}(t^n,z_j)$ and $V_{2}(t^n,z_j)$ can be denoted by $V_{1,j}^{n}$ and $V_{2,j}^{n}$. Besides we write $g^n$ for the value of $G(t^n)$. Let us consider the central approximation of the spatial derivatives,
\begin{equation}\label{central-app1}
\frac{V_{1,j+1}^{n}-V_{1,j-1}^{n}}{2h}\approx\frac{\partial V_{1}}{\partial z}(t^n,z_j),\quad
\frac{V_{2,j+1}^{n}-V_{2,j-1}^{n}}{2h}\approx\frac{\partial V_{2}}{\partial z}(t^n,z_j),
\end{equation}
\begin{equation}\label{central-app2}
\frac{w_{j+1}^{n}-w_{j-1}^{n}}{2h}\approx\frac{\partial w}{\partial z}(t^n,z_j),\quad
\frac{w_{j-1}^{n}-2w_{j}^{n}+w_{j+1}^{n}}{h^2}\approx\frac{\partial^2 w}{\partial z^2}(t^n,z_j).
\end{equation}
By using the forward approximation of the time derivative, we get
\begin{equation}\label{forward-app}
\frac{w_{j}^{n+1}-w_{j}^{n}}{\tau}\approx\frac{\partial w}{\partial t}(t^n,z_j),\quad
\frac{g^{n+1}-g^{n}}{\tau}\approx G^{'}(t^n).
\end{equation}
Let us apply the approximation \eqref{central-app1} on the elliptic equations \eqref{elliptic-eq1}, \eqref{elliptic-eq2}, then we get
\begin{eqnarray}
\frac{1}{G(t)}\frac{1}{h^2}\cdot V_{1,j-1}^n+(\frac{1}{G(t)}\cdot\frac{-2}{h^2}-\lambda_1)V_{1,j}^n+\frac{1}{G(t)}\frac{1}{h^2}\cdot V_{1,j+1}^n=-\mu_{1}w_{j}^n, {\quad 1\leq j\leq M-1},\label{elliptic-eq1-appr}
\\
\frac{1}{G(t)}\frac{1}{h^2}\cdot V_{2,j-1}^n+(\frac{1}{G(t)}\cdot\frac{-2}{h^2}-\lambda_2)V_{2,j}^n+\frac{1}{G(t)}\frac{1}{h^2}\cdot V_{2,j+1}^n=-\mu_{2}w_{j}^n, {\quad 1\leq j\leq M-1}.\label{elliptic-eq2-appr}
\end{eqnarray}
We mainly focus on solving the values of $V_1$, the relevant results about $V_2$ can be obtained in a similar way. For \eqref{central-app1}, $1\leq j\leq M-1$, we can get $M-1$ equations for $V_1$ and there are two other equations we need for the boundary. In order to achieve a higher $O(h^2)$ accuracy (see \cite{Li2000generalized}), we use the idea of finite volume method to handle the boundary conditions.

We get the equation for the left boundary:
\begin{equation}
(\frac{h}{2} d_0+\frac{a_1}{h})V_{1,0}-\frac{a_1}{h} V_{1,1}=\frac{h}{2}\phi_0,
\end{equation}
where
\begin{displaymath}
a_1=(\frac{1}{h}\int_{z_0}^{z_1}-G(t^n)dz)^{-1},\quad d_0=\frac{2}{h}\int_{z_0}^{z_{1/2}}-\lambda_1dz,\quad\phi_0=\frac{2}{h}\int_{z_0}^{z_{1/2}}-\mu_1 w^n dz,
\end{displaymath}
and $z_{1/2}=\frac{z_0+z_1}{2}$.

The equation for the right boundary is similar:
\begin{equation}
-\frac{a_n}{h}V_{1,M-1}+(\frac{a_n}{h}+\frac{h}{2}d_n)V_{1,M}=\frac{h}{2}\phi_n,
\end{equation}
where
\begin{displaymath}
a_n=(\frac{1}{h}\int_{z_{M-1}}^{z_M}-G(t^n)dz)^{-1},\quad d_n=\frac{2}{h}\int_{z_{M-1/2}}^{z_{M}}-\lambda_1dz,\quad\phi_n=\frac{2}{h}\int_{z_{M-1/2}}^{z_{M}}-\mu_1 w^n dz,
\end{displaymath}
and $z_{M-1/2}=\frac{z_{M-1}+z_M}{2}$.

By now we have $M+1$ equations which is enough to form a system of linear algebraic equations for $V_1$. It is tridiagonal and there exists the unique solutions $V_{1,0}^n,V_{1,1}^n,\cdots,V_{1,M}^n$. Similarly, we can acquire the system of linear algebraic equations about $V_2$ and its corresponding solutions $V_{2,0}^n,V_{2,1}^n,\cdots,V_{2,M}^n$.

With the information of $V_1$ and $V_2$, we can concentrate on solving the parabolic equation \eqref{one-free-boundary-eq3}. From \eqref{central-app1}, \eqref{central-app2}, and \eqref{forward-app}, \eqref{one-free-boundary-eq3} is approximated by
\begin{eqnarray}\label{parabolic-app}
&&g^n\frac{w_j^{n+1}-w_j^n}{\tau}-\frac{z_j}{2}\frac{w_{j+1}^n-w_{j-1}^n}{2h}\frac{g^{n+1}-g^n}{\tau}\nonumber
\\
&&=\frac{w_{j-1}^n-2w_{j}^n+w_{j+1}^n}{h^2}+(-\chi_1\frac{V_{1,j+1}^n-V_{1,j-1}^n}{2h}
+\chi_2\frac{V_{2,j+1}^n-V_{2,j-1}^n}{2h})\frac{w_{j+1}^n-w_{j-1}^n}{2h}\nonumber
\\
&&\quad +(-\chi_1\lambda_1V_{1,j}^n+\chi_2\lambda_2V_{2,j}^n+a)w_{j}^ng^n
+(\chi_1\mu1-\chi_2\mu_2-b){(w_j^n)}^2g^n, \label{parabolic-app}
\end{eqnarray}
for $n\geq0,\;0\leq j\leq M-1$. Because of the initial conditions \eqref{initial-cond2}, we can assume a fictitious value $w_{-1}^n$ at the point $(t^n,-h)$, and then
\begin{equation}\label{fictitious-value}
\frac{w_1^n-w_{-1}^n}{2h}=0,\quad w_M^n=0,\quad n\geq 0.
\end{equation}
Considering the \emph{Stefan} condition \eqref{boundary-cond2}, according to \eqref{forward-app} and three points backward spatial approximation of $\frac{\partial w}{\partial z}(t,1)$, we obtain
\begin{equation}\label{boundary-app1}
\frac{g^{n+1}-g^n}{\tau}=-\frac{\nu}{h}(3w_M^n-4w_{M-1}^n+w_{M-2}^n), \quad n\geq 0.
\end{equation}
Because of \eqref{initial-cond1}, it can also be written as:
\begin{equation}\label{boundary-app2}
g^{n+1}=g^n+\frac{\tau\nu}{h}(4w_{M-1}^n-w_{M-2}^n),\quad n\geq0.
\end{equation}
Let us replace $g^{n+1}$ with \eqref{boundary-app2} in \eqref{parabolic-app}, we get the explicit scheme:
\begin{eqnarray}
w_j^{n+1}&=&(-\frac{z_j}{4h}\frac{\tau\nu(4w_{M-1}^n-w_{M-2}^n)}{hg^n}+\frac{\tau}{g^nh^2}-\frac{S_1\tau}{4h^2g^n})w_{j-1}^n\nonumber
\\
&&+(1-\frac{2\tau}{g^nh^2}+S_2\tau)w_j^n+S_3\tau({w_j^n})^2\nonumber
\\
&&+(-\frac{z_j}{4h}\frac{\tau\nu(4w_{M-1}^n-w_{M-2}^n)}{hg^n}+\frac{\tau}{g^nh^2}+\frac{S_1\tau}{4h^2g^n})w_{j+1}^n, \label{parabolic-app2}
\end{eqnarray}
for $ n\geq 0,\;0\leq j\leq M-1,$ where
\begin{eqnarray}
S_1 &=&-\chi_1(V_{1,j+1}^{n}-V_{1,j-1}^{n})+\chi_2(V_{2,j+1}^{n}-V_{2,j-1}^{n}),\label{notation-s1}
\\
S_2 &=& -\chi_1\lambda_1V_{1,j}^n+\chi_2\lambda_2V_{2,j}^n+a,\label{notation-s2}
\\
S_3 &=& \chi_1\mu_1-\chi_2\mu_2-b.\label{notation-s3}
\end{eqnarray}


The solution of \eqref{one-free-boundary-eq3} is classical (see \cite[Lemma 3.1]{BaoShen2}), we can obtain a optimal error estimates for the numerical approximation.
Consider \eqref{one-free-boundary-eq3}, \eqref{boundary-cond1} and \eqref{boundary-cond2}, we denote that

\begin{eqnarray}
L_1(w, V_1, V_2, G)&=&\frac{\partial w}{\partial t}-\frac{z}{2}\frac{G^{'}(t)}{G(t)}\frac{\partial w}{\partial z}-\frac{1}{G(t)}\frac{\partial^2w}{\partial z^2}-\frac{1}{G(t)}(-\chi_1\frac{\partial V_1}{\partial z}+\chi_2\frac{\partial V_2}{\partial z})\frac{\partial w}{\partial z}\nonumber
\\
&&-(-\chi_1\lambda_1V_1+\chi_2\lambda_2V_2 + a)w-(\chi_1\mu_1-\chi_2\mu_2-b)w^2=0,\label{one-free-boundary-eq-app1}
\\
L_2(w, V_1, V_2, G) &=& \frac{\partial w}{\partial z}(t,0)=0,\label{boundary-cond1-app}
\\
L_3(w, V_1, V_2, G) &=& G^{'}(t)+2\nu\frac{\partial w}{\partial z}(t,1)=0.\label{boundary-cond2-app}
\end{eqnarray}

From \eqref{parabolic-app}, we let
\begin{eqnarray}
L_{h1}(w_j^n,V_{1,j}^n,V_{2,j}^n, g^n)&=&\frac{w_j^{n+1}-w_j^n}{\tau}-\frac{z_j}{2}\frac{w_{j+1}^n-w_{j-1}^n}{2h}\frac{g^{n+1}-g^n}{g^n\tau}-\frac{1}{g^n}\frac{w_{j-1}^n-2w_{j}^n+w_{j+1}^n}{h^2}\nonumber
\\
&&-\frac{1}{g^n}(-\chi_1\frac{V_{1,j+1}^n-V_{1,j-1}^n}{2h}+\chi_2\frac{V_{2,j+1}^n-V_{2,j-1}^n}{2h}
)\frac{w_{j+1}^n-w_{j-1}^n}{2h}\nonumber
\\
&&-(-\chi_1\lambda_1V_{1,j}^n+\chi_2\lambda_2V_{2,j}^n+a)w_{j}^n\nonumber
\\
&& - (\chi_1\mu_1-\chi_2\mu_2-b)({w_j^n})^2=0\label{one-free-boundary-eq-app2}
\end{eqnarray}
for $n\geq0,\;0\leq j\leq M-1$. Also from \eqref{fictitious-value} and \eqref{boundary-app1}, let
\begin{eqnarray}
L_{h2}(w_j^n,V_{1,j}^n,V_{2,j}^n, g^n) &=& \frac{w_1^n - w_{-1}^n}{2h} = 0, \quad n\geq 0,\label{boundary-cond1-app-2}
\\
L_{h3}(w_j^n,V_{1,j}^n,V_{2,j}^n, g^n) &=& \frac{g^{n+1} - g^n}{\tau} - \frac{\nu}{h}(4w_{M-1}^n - w_{M-2}^n)=0, \quad n\geq 0.\label{boundary-cond2-app-2}
\end{eqnarray}
Then we have the following error estimates for the parabolic equation.
\begin{prop}(Error estimates)\label{Parabolic-error-estimate}
Under conditions of global existence (\cite[Theorem 1.2]{BaoShen2}), let $w$ be the solution of Equation \eqref{one-free-boundary-eq-app1}, \eqref{boundary-cond1-app}, \eqref{boundary-cond2-app} and $w_h$ be the numerical solution of \eqref{one-free-boundary-eq-app2}, \eqref{boundary-cond1-app-2}, \eqref{boundary-cond2-app-2}. Then we have the following error estimates:
\begin{eqnarray}
R_1(w_j^n,g^n)&=&[L_1(w,G)]_j^n-L_{h1}(w_j^n,g^n)=O(\tau+h^2),
\\
R_2(w_j^n,g^n)&=&[L_2(w,G)]_j^n-L_{h2}(w_j^n,g^n)=O(h^2),
\\
R_3(w_j^n,g^n)&=&[L_3(w,G)]_j^n-L_{h3}(w_j^n,g^n)=O(\tau+h^2).
\end{eqnarray}
\end{prop}
\begin{proof}[The proof of Proposition \ref{Parabolic-error-estimate}]
At the point $(t^n,z_j)$, we have the error estimate
\begin{equation}
\begin{split}
R_1(w_j^n,g^n)&=[L_1(w,G)]_j^n-L_{h1}(w_j^n,g^n)\\
&=\frac{\partial w}{\partial t}(t^n,z_j)-\frac{w_j^{n+1}-w_j^n}{\tau}-\frac{z_j}{2}\frac{G^{'}(t^n)}{G(t^n)}\frac{\partial w}{\partial z}(t^n,z_j)+\frac{z_j}{2}\frac{w_{j+1}^n-w_{j-1}^n}{2h}\frac{g^{n+1}-g^n}{g^n\tau}
\\&-\frac{1}{G(t^n)}\frac{\partial^2w}{\partial z^2}(t^n,z_j)+\frac{1}{g^n}\frac{w_{j-1}^n-2w_{j}^n+w_{j+1}^n}{h^2}\\
&-\frac{1}{G(t^n)}(-\chi_1\frac{\partial V_1}{\partial z}(t^n,z_j)+\chi_2\frac{\partial V_2}{\partial z}(t^n,z_j))\frac{\partial w}{\partial z}(t^n,z_j)
\\&+\frac{1}{g^n}(-\chi_1\frac{V_{1,j+1}^n-V_{1,j-1}^n}{2h}+\chi_2\frac{V_{2,j+1}^n-V_{2,j-1}^n}{2h}
)\frac{w_{j+1}^n-w_{j-1}^n}{2h}
\end{split}
\end{equation}

By using Taylor's expansion at the point $(t^n,z_j)$, we have the following result
\begin{equation}
\label{erroreq1}
\frac{\partial w}{\partial t}(t^n,z_j)-\frac{w_j^{n+1}-w_j^n}{\tau}=O(\tau),
\end{equation}

\begin{equation}
\label{erroreq2}
\begin{split}
&-\frac{z_j}{2}\frac{G^{'}(t^n)}{G(t^n)}\frac{\partial w}{\partial z}(t^n,z_j)+\frac{z_j}{2}\frac{w_{j+1}^n-w_{j-1}^n}{2h}\frac{g^{n+1}-g^n}{g^n\tau}
\\
&=\frac{z_j}{2g^n}(-G^{'}(t^n)\frac{\partial w}{\partial z}(t^n,z_j)+\frac{w_{j+1}^n-w_{j-1}^n}{2h}\frac{g^{n+1}-g^n}{\tau})
\\
&=\frac{z_j}{2g^n}(-G^{'}(t^n)\frac{\partial w}{\partial z}(t^n,z_j)+\frac{g^{n+1}-g^n}{\tau}\frac{\partial w}{\partial z}(t^n,z_j)-\frac{g^{n+1}-g^n}{\tau}\frac{\partial w}{\partial z}(t^n,z_j)
\\
&\quad +\frac{w_{j+1}^n-w_{j-1}^n}{2h}\frac{g^{n+1}-g^n}{\tau})
\\
&=\frac{z_j}{2g^n}\{\frac{\partial w}{\partial z}(t^n,z_j)[\frac{g^{n+1}-g^n}{\tau}-G^{'}(t^n)]
\\
&\quad+[\frac{g^{n+1}-g^n}{\tau}-G^{'}(t^n)+G^{'}(t^n)][\frac{w_{j+1}^n-w_{j-1}^n}{2h}-\frac{\partial w}{\partial z}(t^n,z_j)]\}\\
&=O(\tau+h^2),
\end{split}
\end{equation}

\begin{equation}
\label{erroreq3}
-\frac{1}{G(t^n)}\frac{\partial^2w}{\partial z^2}(t^n,z_j)+\frac{1}{g^n}\frac{w_{j-1}^n-2w_{j}^n+w_{j+1}^n}{h^2}=O(h^2),
\end{equation}

\begin{equation}
\label{errorv1}
\begin{split}
&\frac{\partial V_1}{\partial z}(t^n,z_j)\frac{\partial w}{\partial z}(t^n,z_j)-\frac{V_{1,j+1}^n-V_{1,j-1}^n}{2h}\frac{w_{j+1}^n-w_{j-1}^n}{2h}
\\
&=\frac{\partial V_1}{\partial z}(t^n,z_j)\frac{\partial w}{\partial z}(t^n,z_j)-\frac{\partial V_1}{\partial z}(t^n,z_j)\frac{w_{j+1}^n-w_{j-1}^n}{2h}
\\
&\quad +\frac{\partial V_1}{\partial z}(t^n,z_j)\frac{w_{j+1}^n-w_{j-1}^n}{2h}-\frac{V_{1,j+1}^n-V_{1,j-1}^n}{2h}\frac{w_{j+1}^n-w_{j-1}^n}{2h}\\
&=\frac{\partial V_1}{\partial z}(t^n,z_j)[\frac{\partial w}{\partial z}(t^n,z_j)-\frac{w_{j+1}^n-w_{j-1}^n}{2h}]
\\
&\quad+[\frac{w_{j+1}^n-w_{j-1}^n}{2h}-\frac{\partial w}{\partial z}(t^n,z_j)+\frac{\partial w}{\partial z}(t^n,z_j)][\frac{\partial V_1}{\partial z}(t^n,z_j)-\frac{V_{1,j+1}^n-V_{1,j-1}^n}{2h}]\\
&=O(h^2),
\end{split}
\end{equation}
Similarly, we have
\begin{equation}
\label{errorv2}
\frac{\partial V_2}{\partial z}(t^n,z_j)\frac{\partial w}{\partial z}(t^n,z_j)-\frac{V_{2,j+1}^n-V_{2,j-1}^n}{2h}\frac{w_{j+1}^n-w_{j-1}^n}{2h}=O(h^2),
\end{equation}
 Using estimates \eqref{errorv1} and \eqref{errorv2}, and the linear combinations we obtain
\begin{equation}
\label{erroreq4}
\begin{split}
&-\frac{1}{G(t^n)}[-\chi_1\frac{\partial V_1}{\partial z}(t^n,z_j)+\chi_2\frac{\partial V_2}{\partial z}(t^n,z_j)]\frac{\partial w}{\partial z}(t^n,z_j)
\\&+\frac{1}{g^n}(-\chi_1\frac{V_{1,j+1}^n-V_{1,j-1}^n}{2h}+\chi_2\frac{V_{2,j+1}^n-V_{2,j-1}^n}{2h}
)\frac{w_{j+1}^n-w_{j-1}^n}{2h}=O(h^2),
\end{split}
\end{equation}
Then from \eqref{erroreq1},\eqref{erroreq2},\eqref{erroreq3}, and \eqref{erroreq4}, we have
\begin{equation}
R_1(w_j^n,g^n)=O(\tau+h^2).
\end{equation}
By the central difference approximation, we have the error estimate:
\begin{equation}
R_2(w_j^n,g^n)=[L_2(w,G)]_j^n-L_{h2}(w_j^n,g^n)=O(h^2)
\end{equation}
Using the three points derivative formula \eqref{boundary-app1},
we obtain
\begin{equation}
R_3(w_j^n,g^n)=[L_3(w,G)]_j^n-L_{h3}(w_j^n,g^n)=O(\tau+h^2)
\end{equation}
\end{proof}

From the above results, we have the following error estimate for the transferred system \eqref{one-free-boundary-eq}.
\begin{thm}[Error estimate]\label{consistency}
 Let $L_h(w, V_1, V_2, g)$ denote all the finite difference scheme of the transferred system \eqref{one-free-boundary-eq} including the boundary approximations, and the corresponding continuous scheme as $L(w, V_1, V_2, g)$, then $L_h(w, V_1, V_2, g)$ is consistent with $L(w, V_1, V_2, g)$ and the local truncation error is
\begin{equation}\label{error-est}
T_j^n(w,V_1,V_2, g)=O(h^2+\tau).
\end{equation}
\end{thm}
\begin{proof}
Combining Proposition \ref{Parabolic-error-estimate} for the parabolic equation estimate and the elliptic equation estimate above (central finite difference for the elliptic equation and finite volume scheme for the boundary), we have the local truncation error estimate \eqref{error-est}.
\end{proof}
Before proving the positivity of the density function $w$ and the monotonicity property of the spreading front $g^n$, we need the following maximum principle of the discrete elliptic equations.
\begin{prop}
Consider equations \eqref{elliptic-eq1-appr} and \eqref{elliptic-eq2-appr}, we have the following discrete maximum principle
\begin{equation}\label{discrete-max-princ}
\frac{\mu_i}{\lambda_i}w_{\min}^n\leq  V_{i,j}^n\leq \frac{\mu_i}{\lambda_i}w_{\max}^n,\quad i=1,2,\quad j=1,2\dots,M-1,
\end{equation}
where $w^n_{\min}, w^n_{\max}$ are the minimum and maximum solution of Equation \eqref{parabolic-app}.
\end{prop}
\begin{proof}
We only study the discrete maximum principle of Equation \eqref{elliptic-eq1-appr}, the proof for Equation \eqref{elliptic-eq2-appr} is the same. Rewrite Equation \eqref{elliptic-eq1-appr} as the following
\begin{equation}
V_{1,j-1}^n + V_{1,j+1}^n + 2h\nu_1G(t) w_j^n = 2(1 + \lambda_1 hG(t))V_{1,j}^n.
\end{equation}
Let $j=j^*(j^*\neq 0,M)$ be the point where $V_1$ takes the maximum at $t^n$ and denote it as $V_{1,max}^n$, then
$$
2h\mu_1G(t)w_{j^*}^n=2\lambda_1hG(t)V_{1,max}^n+(V_{1,max}^n-V_{1,j^*-1}^n)+(V_{1,max}^n-V_{1,j^*+1}).
$$
Because $V_{1,max}^n-V_{1,j^*-1}^n\geq0$, $V_{1,max}^n-V_{1,j^*+1}\geq0$, and $2h\mu_1G(t)>0$, we have
$$
2h\mu_1G(t)w_{max}^n\geq2h\mu_1G(t)w_{j^*}^n\geq2\lambda_1hG(t)V_{1,max}^n,
$$
where $w_{max}^n$ is the maximum of $w$ at $t^n$. We then conclude that
$$V_{1,max}^n\leq \frac{\mu_1}{\lambda_1}w_{max}^n.$$
Similarly we can obtain
\begin{eqnarray*}
V_{2,max}^n \leq \frac{\mu_2}{\lambda_2}w_{max}^n,\quad
V_{1,min}^n \geq \frac{\mu_1}{\lambda_1}w_{min}^n,\quad
V_{2,min}^n \geq \frac{\mu_2}{\lambda_2}w_{min}^n.
\end{eqnarray*}
\end{proof}
The moving boundary $h(t)$ is monotonicity in the system \eqref{one-free-boundary-eq} which is also preserved in the discretized one.

\begin{thm}[Monotonicity of spreading front]\label{Increasing-S-front}
Let $\chi_1,\chi_2$ be small enough and $\tau$ satisfy
\begin{equation}\label{condition-tau-1}
\tau<\frac{h^2}{\frac{\mu C}{g^0}+h^2(bC-a)}
\end{equation}
then $w^{n}_{M-1}>0,$ and $g^n$ is monotonicity in $n$ with $n\geq0, C=e^{aT}w_{M-1}^0$.
\end{thm}
\begin{proof}
We prove the positivity and monotonicity of the free boundary $g^n$ by using the induction principle on the index $n$. For $n=0$, from the initial condition of $w_0(z)$, we have $w_j^0>0,\;0\leq j\leq M-1,w_M^0=0$. Additionally, by Hopf lemma the left derivative of $w^0$ at $z_M$ is negative and hence the corresponding difference approximation of Equation \eqref{boundary-app1} with small enough $h$ has
\begin{equation}
(3w_M^n-4w_{M-1}^n+w_{M-2}^n)<0,
\end{equation}
which is equivalent to $(w_{M-2}^0-4w_{M-1}^0)<0$. So combined with $g^0>0$ and \eqref{boundary-app1} with small enough $h$, we have
\begin{equation}
g^1>g^0>0.
\end{equation}

By using Taylor's expansion on the left of $z_M=1$ at $t^n$, we have
\begin{equation}
w=w_M^n+\frac{w_M^n-w_{M-1}^n}{h}(z-1)+O(h^2).
\end{equation}
Let $w=w_{M-2}^n$, then we have
\begin{equation}
w_{M-2}^n=w_M^n+\frac{w_M^n-w_{M-1}^n}{h}\cdot(-2h)+O(h^2),
\end{equation}
combined with $w_M^n=0$, we obtain
\begin{equation}
\label{eqwM-2}
w_{M-2}^n=2w_{M-1}^n+O(h^2),\quad n\geq 0.
\end{equation}
Plugging \eqref{eqwM-2} into \eqref{boundary-app1}, we have
\begin{equation}\label{gn2}
g^{n+1}=g^n+\frac{\tau\nu}{h}(2w_{M-1}^n+O(h^2)).
\end{equation}
If $w_{M-1}^n$ is positive, then from \eqref{gn2} $g^n$ is positive and increasing with $n$.
By using \eqref{eqwM-2} and let $j=M-1$, from \eqref{parabolic-app2}  we have
\begin{equation}
\label{eqwM-1}
\begin{split}
w_{M-1}^{n+1}&=(1+\tau(a-bw_{M-1}^n)-\frac{\tau}{h^2}\frac{z_{M-1}}{g^n}\nu w_{M-1}^n)w_{M-1}^n\\
&-(-\chi_1(V_{1,M}^n-V_{1,M-2}^n)+\chi_2(V_{2,M}^n-V_{2,M-2}^n))\frac{\tau}{h^2}\frac{1}{2g^n}w_{M-1}^n\\
&+(-\chi_1\lambda_1V_{1,M-1}^n+\chi_2\lambda_2V_{2,M-1}^n)\tau w_{M-1}^n\\
&+(\chi_1\mu_1-\chi_2\mu_2)\tau({w_{M-1}^n})^2+O(h^2)
\end{split}
\end{equation}
In order to preserve the stability in the forward approximation of parabolic equation, one needs the requirement such that $\frac{\tau}{h^2}<\frac{1}{2}$, and from the discrete maximum principle the bound of $V_1^n$ and $V_2^n$ are controlled by $w^n$ which is bounded by iteration. Assume $\chi_1$ and $\chi_2$ are small enough, we have the approximation
\begin{equation}\label{w-app}
w_{M-1}^{n+1}\approx (1+\tau(a-bw_{M-1}^n)-\frac{\tau}{h^2}\frac{z_{M-1}}{g^n}\nu w_{M-1}^n)w_{M-1}^n.
\end{equation}
In the following, we prove that $w_{M-1}^n$ has a up bound independent of $n$ (n depends on T).
By induction, we assume that $g^n>g^{n-1}>\cdots>g^1>g^0$. With small enough parameters $\tau,h,\chi_1,\chi_2$, and $g^n>g^0,z_{M-1}<1$ we have
\begin{eqnarray}
w_{M-1}^{n+1}&\approx& (1+\tau(a-bw_{M-1}^n)-\frac{\tau}{h^2}\frac{z_{M-1}}{g^n}\nu w_{M-1}^n)w_{M-1}^n\nonumber
\\
&>&(1+\tau(a-bw_{M-1}^n)-\frac{\tau}{h^2}\frac{1}{g^0}\nu w_{M-1}^n)w_{M-1}^n\nonumber
\\
&=&\phi_{M-1}^nw_{M-1}^n \label{w-est}
\end{eqnarray}
where $\phi_{M-1}^n=1+\tau(a-bw_{M-1}^n)-\frac{\tau}{h^2}\frac{1}{g^0}\nu w_{M-1}^n$. From \eqref{w-app} we obtain $w_{M-1}^{n+1}<(1+\tau a)w_{M-1}^n$ and
\begin{equation}
w_{M-1}^n<(1+\tau a)w_{M-1}^{n-1}<\cdots<(1+\tau a)^nw_{M-1}^0
\end{equation}
Denote the total time from $t^0$ to $t^{n+1}$ as $T$ and we have $T=(n+1)\tau$, then
\begin{equation}
(1+\tau a)^n<(1+\tau a)^{n+1}\leq e^{a(n+1)\tau}=e^{aT},
\end{equation}
which leads to $w_{M-1}^n<e^{aT}w_{M-1}^0$. Furthermore, if the time step $\tau$ satisfies
\begin{equation}
\tau<\frac{h^2}{\frac{\mu C}{g^0}+h^2(bC-a)},
\end{equation}
we obtain $\phi^n_{M-1}>0$ and $w^{n+1}_{M-1}>0$ where $C=e^{aT}w_{M-1}^0$. Then combined Equation \eqref{gn2} with the positivity of $w^n_{M-1}$, we conclude the monotonicity of the spreading front $g^n$.
\end{proof}

\begin{thm}[Positivity and boundedness of the discrete solution]\label{Positivity}
In Equation \eqref{one-free-boundary-eq-app2}, let $\tau$ satisfy
\begin{equation}
\label{tau-condition}
\tau < \tau_0=\min\{\frac{h^2}{\frac{\mu C}{g^0}+h^2(bC-a)},\frac{h^2}{\frac{2}{g^0}+h^2((2\chi_2\mu_2+b)e^{(a+U(\chi_1\mu_1+\chi_2\mu_2))T}|w^{0}_{max}|-a)}\},
\end{equation}
 with small $\chi_1,\chi_2$ and let $M$ large enough which is equivalent to small $h$, then the solution of \eqref{one-free-boundary-eq-app2} satisfies $w_j^n\geq 0$ with uniform up-bound,\quad for $0\leq j\leq M,n\geq 0$.
\end{thm}
\begin{proof}
From Equations \eqref{parabolic-app2}, \eqref{notation-s1}, \eqref{notation-s2}, \eqref{notation-s3}, we can rewrite $w_j^{n+1}$ as
\begin{equation}
w_j^{n+1}=a_j^nw^n_{j-1}+b_j^nw^n_{j}+c_j^nw^n_{j+1},
\end{equation}
where
\begin{eqnarray}
a_j^n &=&\frac{\tau}{h^2}(\frac{1}{g^n}-\frac{z_j\nu(4w^n_{M-1}-w^n_{M-2})-S_1}{4g^n}),\label{notation-a}
\\
b_j^n &=& 1-\frac{2\tau}{g^nh^2}+S_2\tau+S_3\tau w_j^n,\label{notation-b}
\\
c_j^n &=& \frac{\tau}{h^2}(\frac{1}{g^n}+\frac{z_j\nu(4w^n_{M-1}-w^n_{M-2})+S_1}{4g^n}).\label{notation-c}
\end{eqnarray}
If $a_j^n,b_j^n,c_j^n$ are positive, the positivity of $w^n_j, n\geq 0, 0\leq j\leq M$ can be proved by induction.
First, we consider
\begin{equation*}
a_j^n=\frac{\tau}{h^2}(\frac{1}{g^n}-\frac{z_j\nu(4w^n_{M-1}-w^n_{M-2})-S_1}{4g^n})
\end{equation*}
From the discrete maximum principle, regularity of $w$ at $z_M = 1 (w^n_M=0),$ by taking large enough $M$ or small enough $h$ and small enough chemotactic sensitivity $\chi_1,\chi_2$, then plugging \eqref{eqwM-2} into \eqref{notation-a} and combined with the positivity of $w^n_{M-1}$ and $0\leq z_j<1$, we have
\begin{equation}
\label{condition-an1}
w^n_{M-1}< \frac{2}{\nu}-\frac{S_1}{2z_j\nu},
\end{equation}
with negative $S_1$, or stronger condition
\begin{equation}\label{condition-an2}
w^n_{M-1} <\frac{2}{\nu}-\frac{S_1}{2h\nu},
\end{equation}
with positive $S_1$. Both conditions lead to $a_j^n>0$.

Next, we study the positivity of $c_j^n$.
\begin{equation}
\label{eqcn}
c_j^n=\frac{\tau}{h^2}(\frac{1}{g^n}+\frac{z_j\nu(4w^n_{M-1}-w^n_{M-2})+S_1}{4g^n}).
\end{equation}
Plugging \eqref{eqwM-2} into \eqref{eqcn} and combine with the positivity of $w^n_{M-1}$ and small enough $\chi_1,\chi_2$, and $h$, the positivity of $c_j^n$ can be guaranteed.

Thirdly, we investigate the positivity of
\begin{equation}
\label{eqbn}
b_j^n=1-\frac{2\tau}{g^nh^2}+S_2\tau+S_3\tau w_j^n.
\end{equation}
In order to obtain $b_j^n>0$, we need
\begin{equation}
\label{position-cond1}
\tau<\frac{h^2}{\frac{2}{g^n}-h^2(S_2+S_3w_j^n)}.
\end{equation}
(While $\chi_1,\chi_2$ are small enough, $\tau(\frac{2}{g^nh^2}-S_2-S_3w_j^n)>0$.) By the discrete maximum principle of the elliptic equation $V_1$ and $V_2$, the monotonicity of $g^n$ $(g^0<g^n),$ we can improve the requirement
\eqref{position-cond1} to
\begin{equation}
\label{position-cond2}
\tau<\min\{\frac{h^2}{\frac{\mu C}{g^0}+h^2(bC-a)},\frac{h^2}{\frac{2}{g^0}+h^2((2\chi_2\mu_2+b)|w^n_{max}|-a)}\},
\end{equation}
where $|w^n_{max}|=\max_{0\leq j\leq M}|w^n_j|$, which leads to the positivity of $b_j^n$.

The positivity of $w_j^n, 0\leq j\leq M, n\geq 0$ is followed from the positivity of $a_j^n,b_j^n,c_j^n$.

For the uniform up-bound, we use the induction method again. By using discrete maximum principle of $V_1^n$ and $V_2^n$, whose values are controlled by maximum of $w^n$, $S_3<0$ and positivity of $a_j^n,b_j^n,c_j^n$, we have estimate in \eqref{parabolic-app2} such that
\begin{equation}
\label{inequality1}
\begin{split}
w_j^{n+1}&\leq(1+\tau(S_2+S_3w_j^n))|w^n_{max}|\\
&\leq(1+\tau((\chi_1\mu_1+\chi_2\mu_2)|w^n_{max}|+a))|w^n_{max}|\\
&\leq(1+\tau a)|w^n_{max}|+\tau(\chi_1\mu_1+\chi_2\mu_2)|w^n_{max}|^2\\
\end{split}
\end{equation}
Because of the existence of up bound for $w^n$, there exist a constant $U$ such that $|w^n_{max}|\leq U$. Furthermore, we have the following uniform up bound estimate
\begin{equation}
\label{inequality2}
\begin{split}
w_j^{n+1}&\leq(1+\tau(a+U(\chi_1\mu_1+\chi_2\mu_2)))|w^n_{max}|\\
&\leq(1+\tau(a+U(\chi_1\mu_1+\chi_2\mu_2)))^2|w^{n-1}_{max}|\\
&\leq\cdots\leq(1+\tau(a+U(\chi_1\mu_1+\chi_2\mu_2)))^{n+1}|w^{0}_{max}|.
\end{split}
\end{equation}
Denote the total time $T=n\tau$, we have
\begin{equation}\label{Uniform-bound}
|w^n_{max}|\leq e^{(a+U(\chi_1\mu_1+\chi_2\mu_2))T}|w^{0}_{max}|,
\end{equation}
and the requirement of \eqref{position-cond2} can be sharped to
\begin{equation*}
\tau < \tau_0=\min\{\frac{h^2}{\frac{\mu C}{g^0}+h^2(bC-a)},\frac{h^2}{\frac{2}{g^0}+h^2((2\chi_2\mu_2+b)e^{(a+U(\chi_1\mu_1+\chi_2\mu_2))T}|w^{0}_{max}|-a)}\},
\end{equation*}
which guarantee the positivity of the discrete solution $w^n$ for $0\leq j\leq M,n\geq 0$.
\end{proof}

Before discussing about the stability of our algorithm and for the sake of clarity in the presentation we specify the concept of stability we use below. We recall the definition of the supremum norm of a vector $x=(x_1,x_2,\cdots,x_n)^T$ in $\mathbb{R}^n$ as $\|x\|_{\infty}=\max(|x_1|,|x_2|,\cdots,|x_n|).$
\begin{defn}
In Equation \eqref{one-free-boundary-eq-app2}, the numerical scheme is said to be $\|\cdot\|_{\infty}$ stable in the domain $[0,T]\times[0,1]$, if for every partition with $T=N\tau,Mh=1$ it hold true that:
\begin{equation}
\|w^n\|_{\infty}\leq K\|w^0\|_{\infty},\quad 0\leq n\leq N,
\end{equation}
where $w^n=[w_0^n,w_1^n,\cdots,w_M^n]^T$is the vector solution of the scheme at $t^n$, $K$ is a constant independent of $h,\tau,n$.
\end{defn}
From the results in Theorem \ref{Positivity} and Estimate \eqref{Uniform-bound}, we have the uniform stability result.
\begin{thm}[Stability of discrete solution]\label{stability}
In Equations \eqref{elliptic-eq1-appr}, \eqref{elliptic-eq2-appr} and \eqref{one-free-boundary-eq-app2}, let $\tau$ satisfy
\begin{equation}
\label{tau-condition}
\tau < \tau_0=\min\{\frac{h^2}{\frac{\mu C}{g^0}+h^2(bC-a)},\frac{h^2}{\frac{2}{g^0}+h^2((2\chi_2\mu_2+b)e^{(a+U(\chi_1\mu_1+\chi_2\mu_2))T}|w^{0}_{max}|-a)}\},
\end{equation}
 with small $\chi_1,\chi_2$ and let $M$ large enough which is equivalent to small $h$, then the discrete solution of \eqref{one-free-boundary-eq} is $\|\cdot\|_{\infty}$ stable.
\end{thm}

\begin{rk}
In the investigation, \eqref{tau-condition} is a strong requirement and it is a sufficient condition for the positivity and stability. However, our simulations indicate larger $h, \tau$ can also guarantee the positivity and stability of the scheme.
\end{rk}

\section{Numerical experiments}
In this section, we study the numerical simulations of the free boundary problem of \eqref{one-free-boundary-eq} in the case that $a(t,x) = 2, b(t,x) = 1$. Our theoretical results indicate there exists a critical value $l^* = \frac{\pi}{2}\sqrt{\frac{1}{a}}$, which is independent of the chemotactic sensibility coefficients $\chi_1, \chi_2$, such that spreading of the species is guaranteed for $h_0 \geq l^*$ and vanishing happens for $h_0 < l^*$. Furthermore, in order to obtain local convergence and persistency, chemotactic sensitivity must be small enough.

Compared to \emph{Fisher-KPP} free boundary problems, $\chi_1=\chi_2=0$ in the system \eqref{one-free-boundary-eq}, even if $h_0 < l^*$, the spreading is guaranteed under condition $\nu > \nu^*>0$ or $\sigma>\sigma^*>0$, where $u_0(x) =\sigma \phi(x)$ and $\nu^*$ is an unknown threshold depending on $u_0$ (see \cite[Theorem 3.9]{DuLi}). Because of the lack of comparison principle in Chemotaxis system \eqref{one-free-boundary-eq}, the existence of $\nu^*$ and $\sigma^*$ is still an open problem. However, our numerical simulation do validate such existence. Furthermore, when the spreading happens in \emph{Fisher-KPP} free boundary problem, we have the following asymptotic spreading speed result:
\begin{equation*}
 \lim_{t\to\infty}\frac{h(t,u_0,h_0)}{t} = c^*
\end{equation*}
where $c^*$ is depending on the logistic coefficient $a(t,x)$ (see \cite{DuLi}, \cite{du2013pulsating}, and \cite{Li2016duffusive2}). Such result is also confirmed in our numerical simulations and the theoretical analysis will be under our investigation in the future.

In the following, we show different simulation results depending on different parameter selections. All parameters satisfy conditions {\bf (H1)}-{\bf (H3)}.
\subsection{Numerical vanishing-spreading dichotomy}
Compared to the vanishing-spreading dichotomy in \emph{Fisher-KPP} free boundary problems, vanishing happens when the initial habitat $h_0$, initial solution $u_0(x) =\sigma \phi(x)$, and moving speed $\nu$ are small enough; spreading happens when either $h_0, u_0(x),$ or $\mu$ is big enough. However, because of the lack of comparison principle, similar results are still open in the chemotaxis system \eqref{one-free-boundary-eq}. The following numerical simulations validate these similar results in chemotaxis system.
\begin{Ex}\label{Exm-1}
In the logistic chemotaxis model \eqref{one-free-boundary-eq}, let $h_0 = 2.5 > l^* =1.11, u_0 = \cos(\pi x/2h_0)$ and
$(\chi_1, \chi_2, \nu, \lambda_1,\lambda_2,\mu_1,\mu_2) = (0.02, 0.01, 0.01, 2, 1, 2, 1)$. Figure \ref{Fig:1} and \ref{Fig:2} show the system has spreading tendency and the asymptotic speed $\frac{h(t)}{t}$ converges to a constant.
\begin{figure}[!htb]
   \begin{minipage}{0.48\textwidth}
     \centering
     \includegraphics[width=.9\linewidth,height=4.0cm]{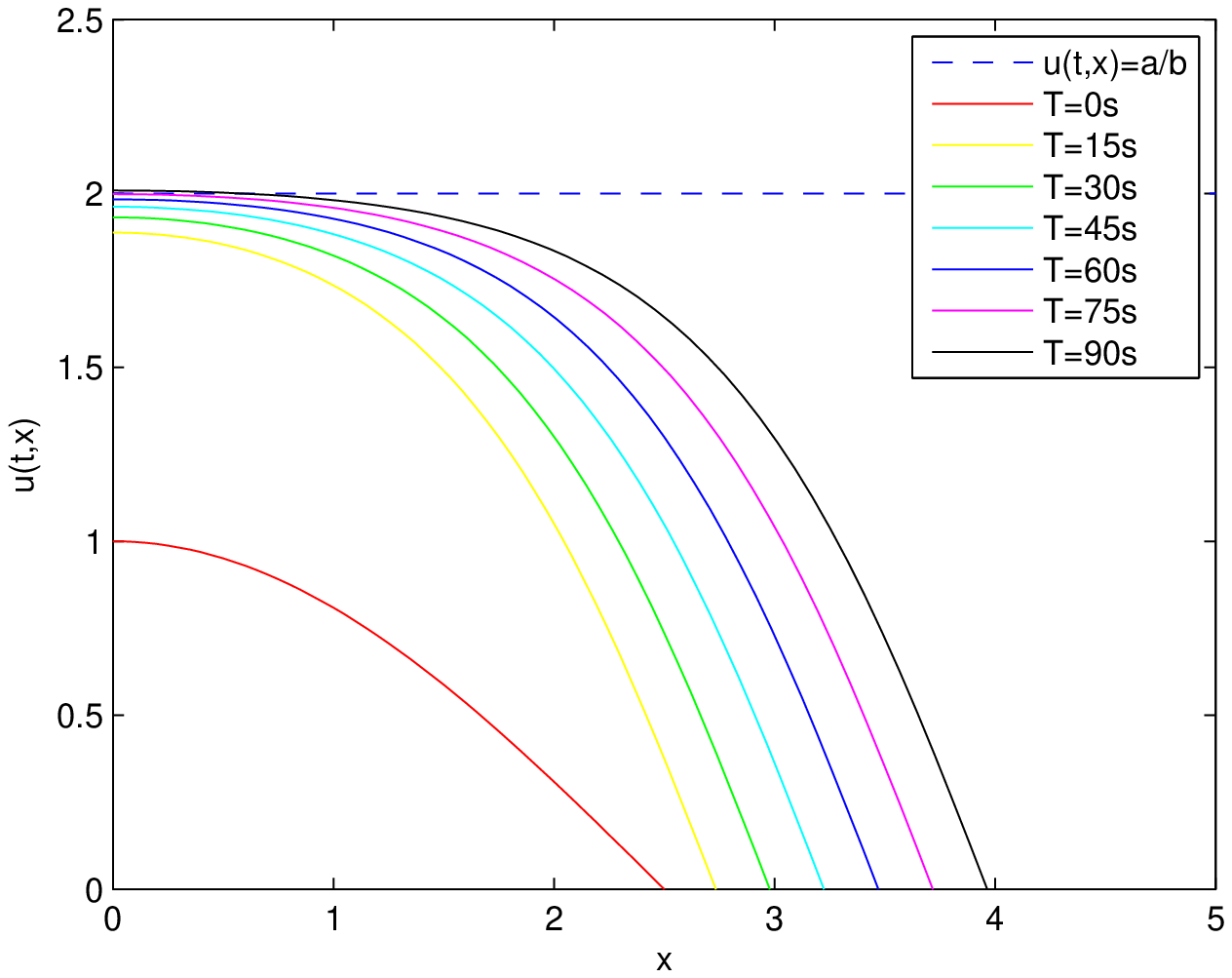}
     \caption{Evolution of the density $u(t,x)$}\label{Fig:1}
   \end{minipage}\hfill
   \begin {minipage}{0.48\textwidth}
     \centering
     \includegraphics[width=.9\linewidth,height=4.0cm]{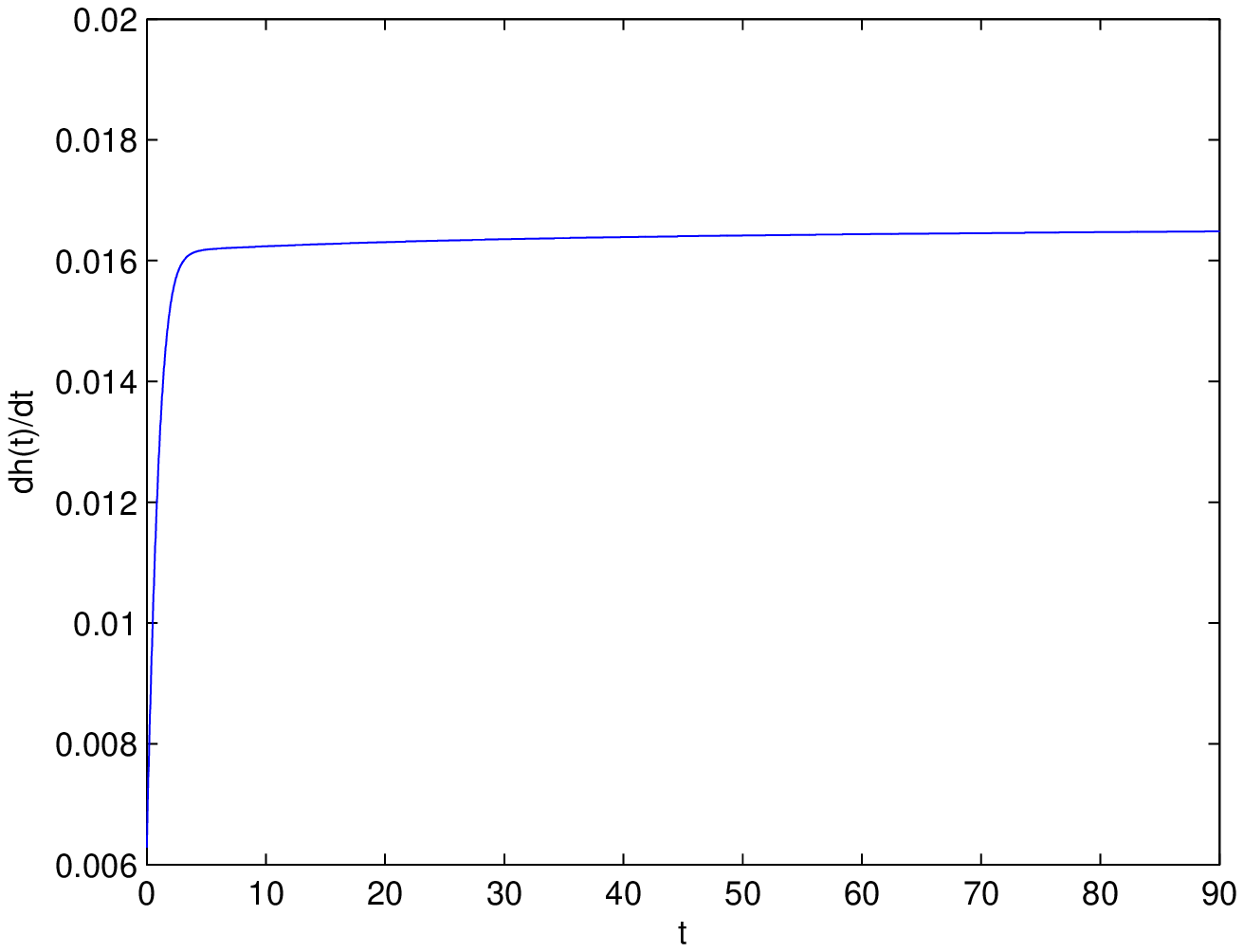}
     \caption{Evolution of the speed $\frac{h(t)}{t}$}\label{Fig:2}
   \end{minipage}
  \end{figure}
\end{Ex}

\begin{Ex}\label{Exm-2}
In the logistic chemotaxis model \eqref{one-free-boundary-eq}, let $h_0 = 0.5 < l^* =1.11, u_0 = \cos(\pi x/2h_0)$, and
$(\chi_1, \chi_2, \nu, \lambda_1,\lambda_2,\mu_1,\mu_2) = (2, 1, 0.8, 1, 2, 1, 2)$. Figure \ref{Fig:3} and \ref{Fig:4} show the system has vanishing tendency and the asymptotic speed $\frac{h(t)}{t}$ converges to zero.
 \begin{figure}[!htb]
   \begin{minipage}{0.48\textwidth}
     \centering
     \includegraphics[width=.9\linewidth,height=4.0cm]{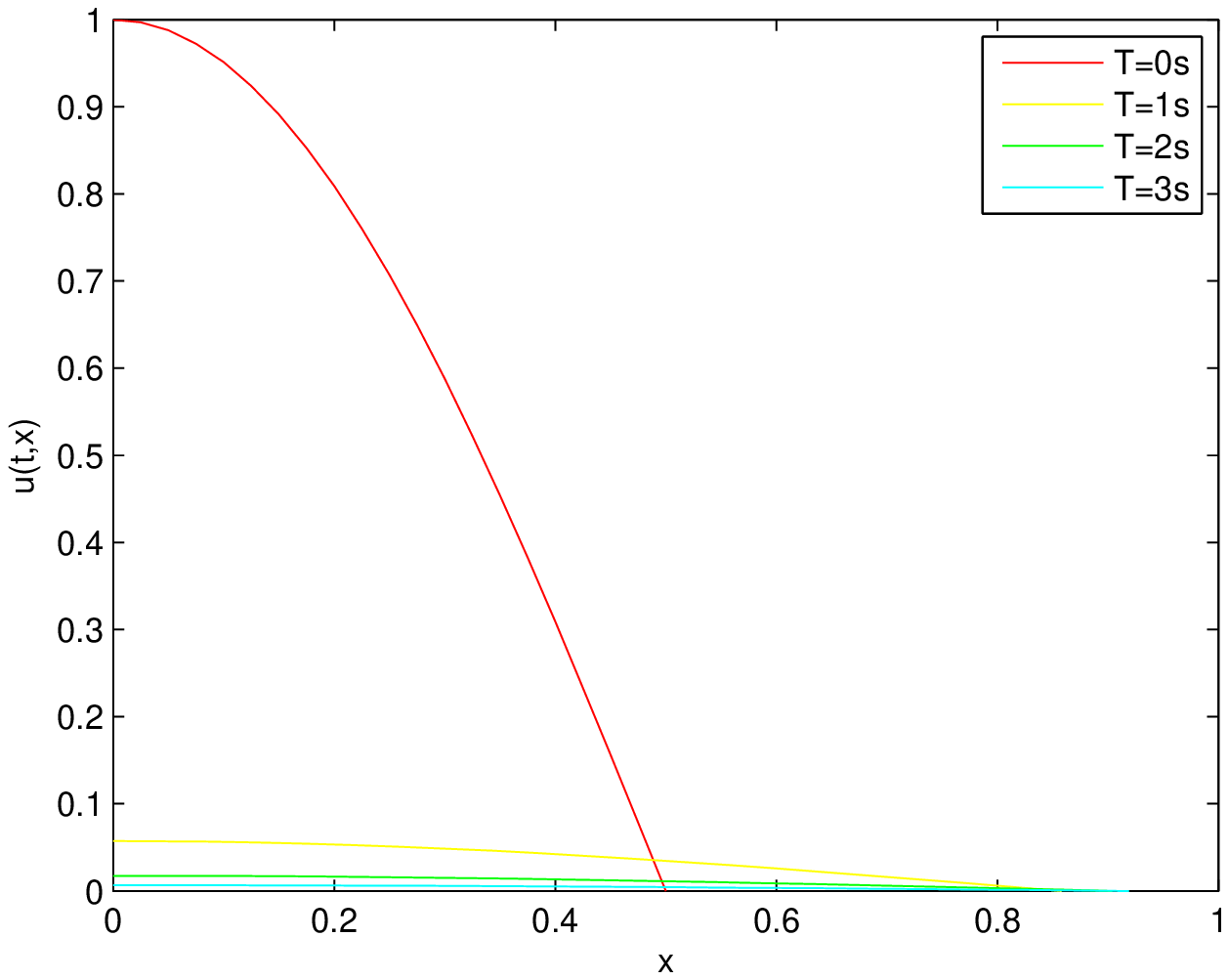}
     \caption{Evolution of the density $u(t,x)$}\label{Fig:3}
   \end{minipage}\hfill
   \begin {minipage}{0.48\textwidth}
     \centering
     \includegraphics[width=.9\linewidth,height=4.0cm]{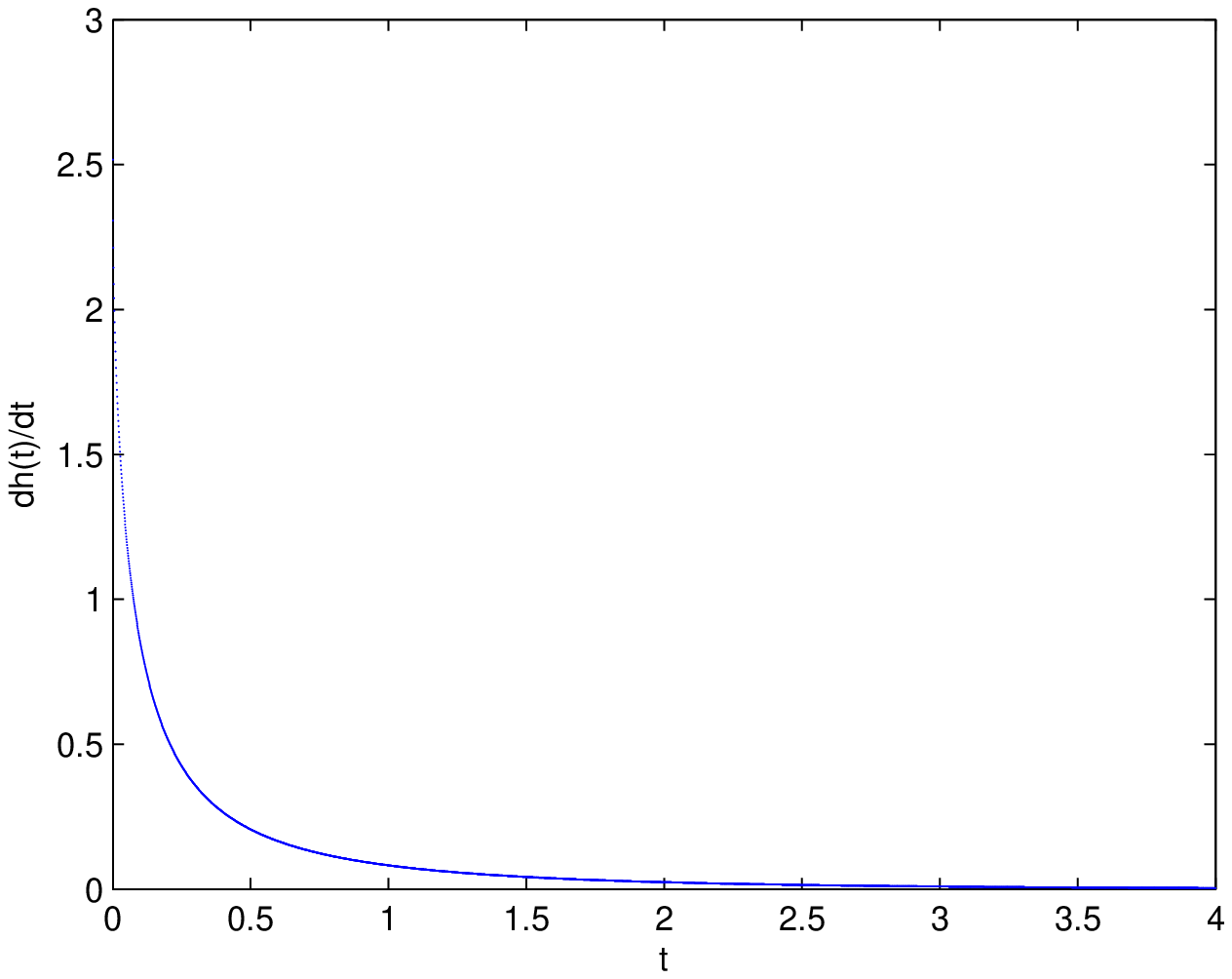}
     \caption{Evolution of the speed $\frac{h(t)}{t}$}\label{Fig:4}
   \end{minipage}
\end{figure}
\end{Ex}

\begin{Ex}\label{Exm-3}
Increase the initial habitat $h_0$ in Example \ref{Exm-2} such that $h_0=2.5>l^*=1.11$, and let $(\chi_1, \chi_2, \nu, \lambda_1,\lambda_2,\mu_1,\mu_2) = (2, 1, 0.8, 1, 2, 1, 2)$. Figure \ref{Fig:5} and \ref{Fig:6} show the system has spreading tendency and the asymptotic speed $\frac{h(t)}{t}$ converges to a positive constant.
\begin{figure}[!htb]
   \begin{minipage}{0.48\textwidth}
     \centering
     \includegraphics[width=.9\linewidth,height=4.0cm]{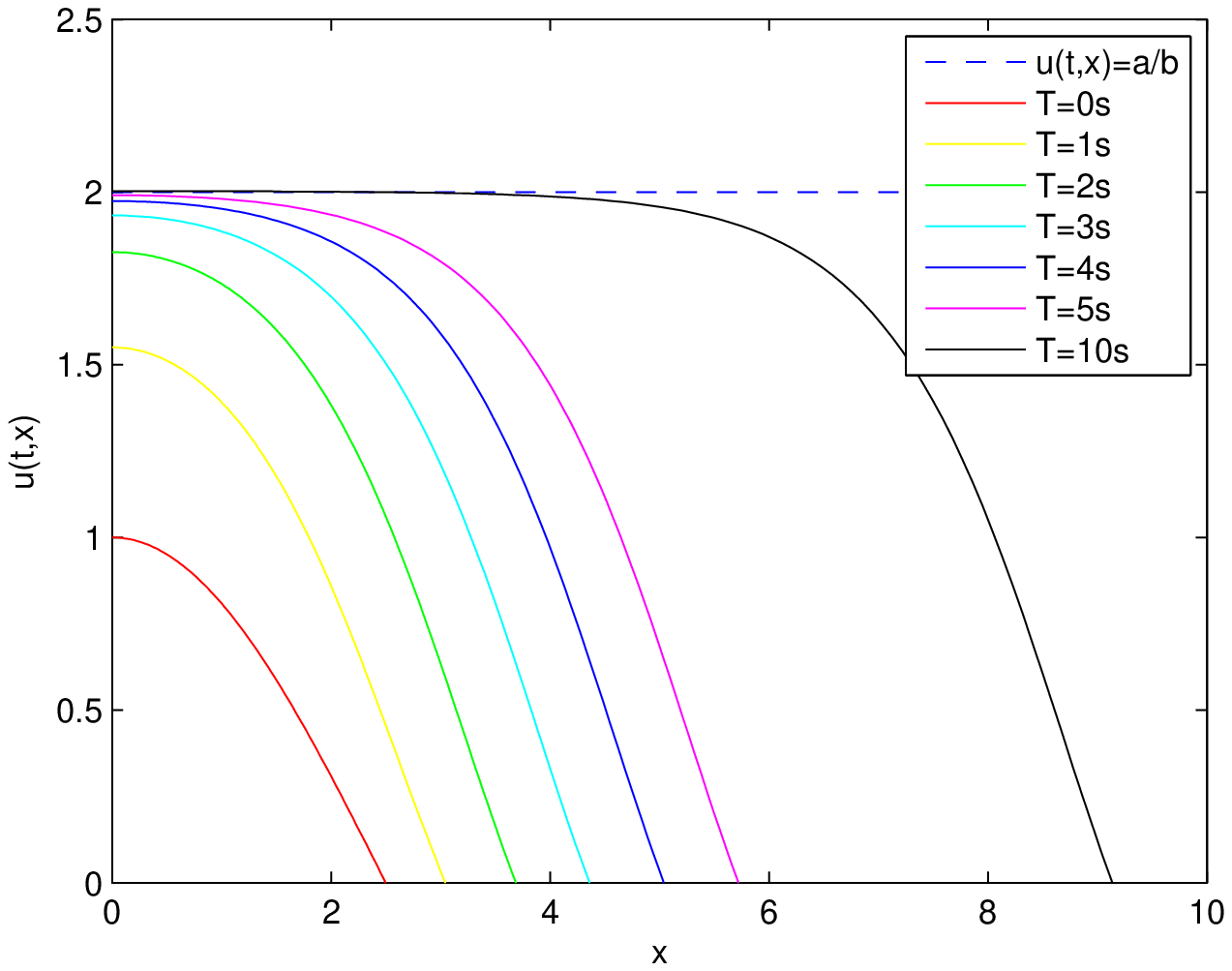}
     \caption{Evolution of the density $u(t,x)$}\label{Fig:5}
   \end{minipage}\hfill
   \begin {minipage}{0.48\textwidth}
     \centering
     \includegraphics[width=.9\linewidth,height=4.0cm]{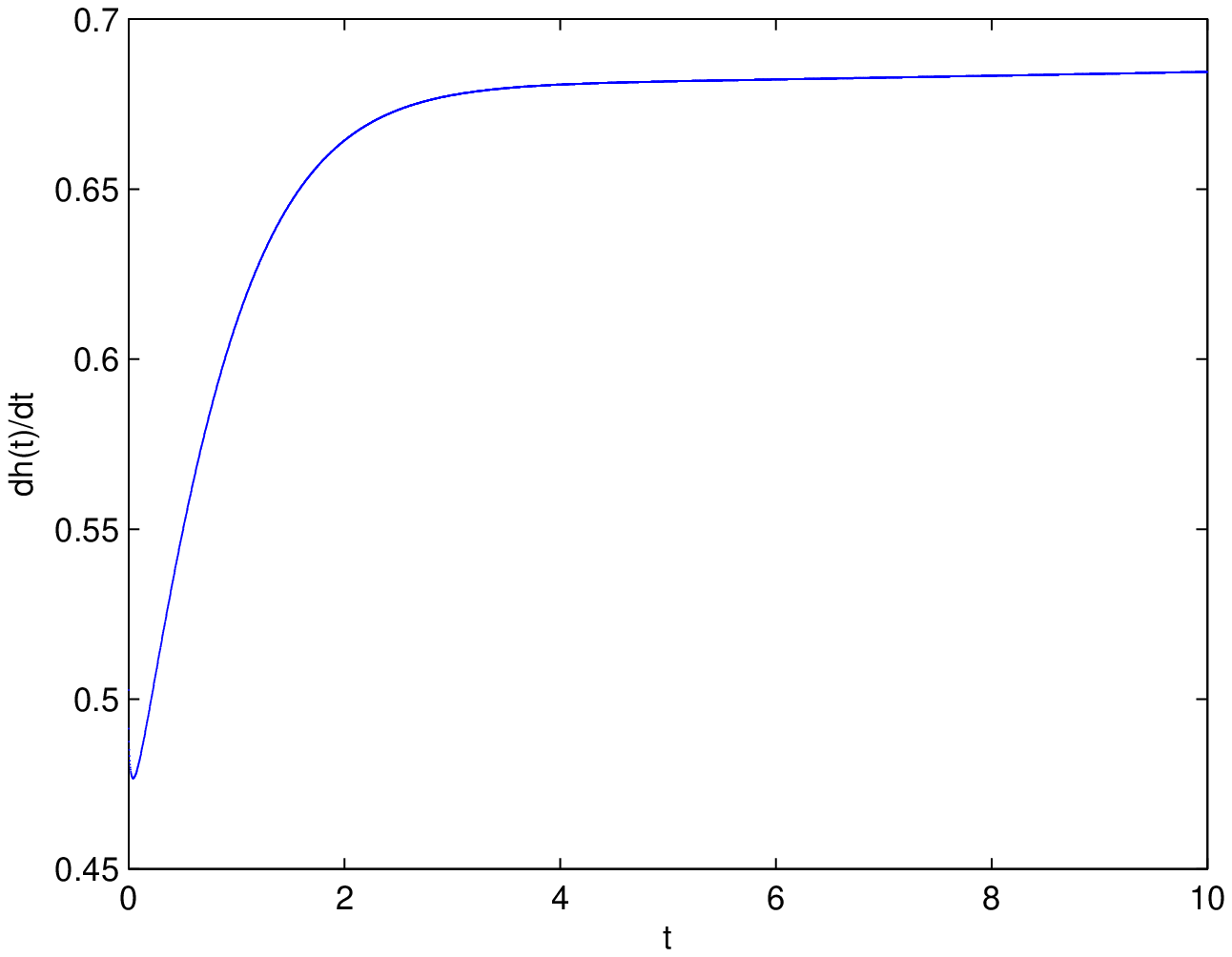}
    \caption{Evolution of the speed $\frac{h(t)}{t}$}\label{Fig:6}
   \end{minipage}
\end{figure}
\end{Ex}
\subsection{Critical value of the moving speed $\nu^*$}
In the following, we study the influence of the moving speed $\nu$ on the vanishing-spreading dichotomy. Compared to \emph{Fisher-KPP} free boundary problem, there exists a critical value $\nu^*$ such that vanishing happens when $\nu < \nu^*$ and spreading happens in other direction. However, because of the lack of comparison principle, whether there exists a critical value $\nu^*$ in chemotaxis free boundary problem is still an open question, but numerically we find the existence of such $\nu^*$ by dichotomy method.
\begin{Ex}\label{Exm-4}
In the logistic chemotaxis model \eqref{one-free-boundary-eq}, let $h_0 = 1.0 < l^* =1.11, u_0 = \cos(\pi x/2h_0)$ and
$(\chi_1, \chi_2, \nu, \lambda_1,\lambda_2,\mu_1,\mu_2) = (0.2, 0.1, 2, 1, 2, 1, 2)$. Figure \ref{Fig:7} and \ref{Fig:8} show the system has a spreading tendency and the asymptotic speed $\frac{h(t)}{t}$ converges to a positive constant.
\begin{figure}[!htb]
   \begin{minipage}{0.48\textwidth}
     \centering
     \includegraphics[width=.9\linewidth,height=4.0cm]{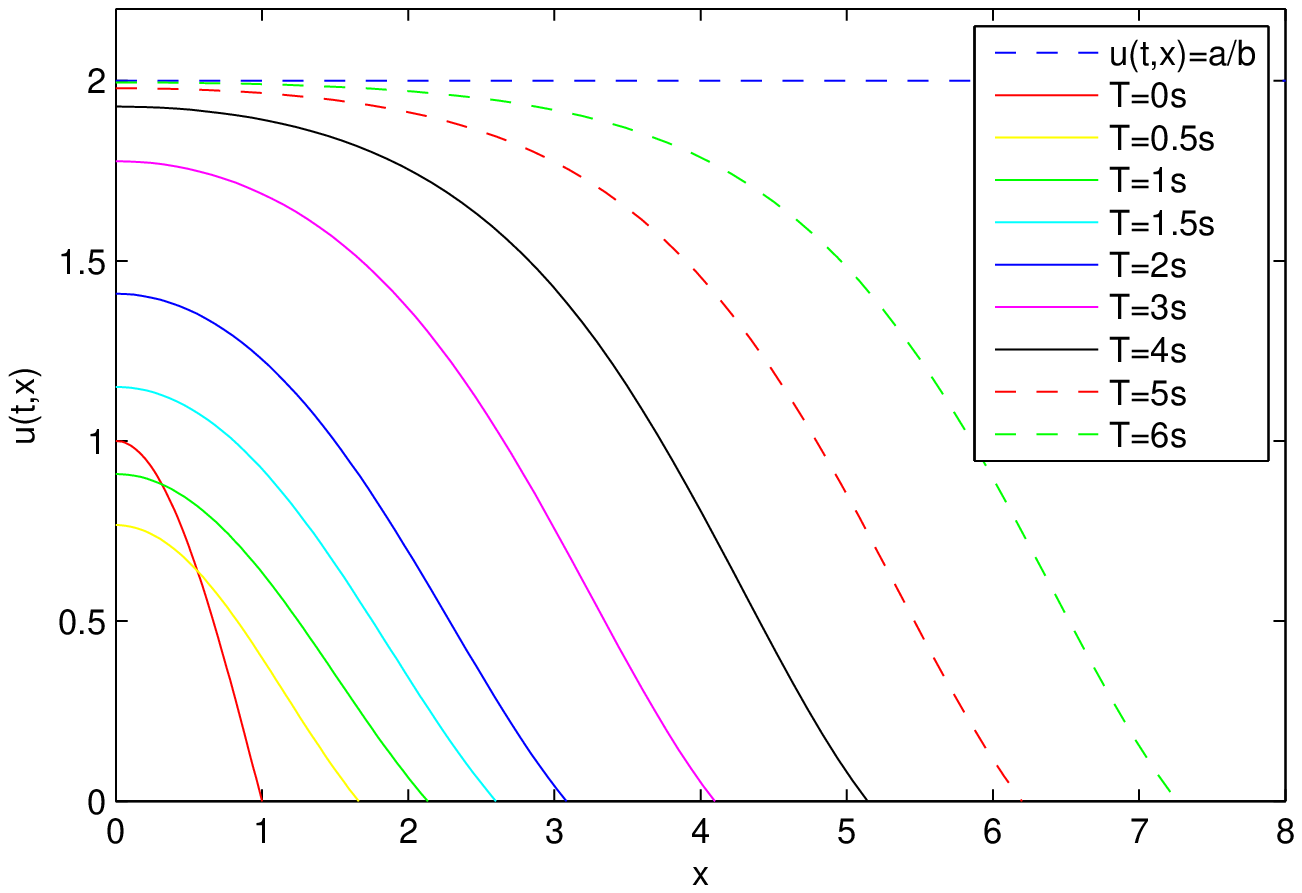}
     \caption{Evolution of the density $u(t,x)$}\label{Fig:7}
   \end{minipage}\hfill
   \begin {minipage}{0.48\textwidth}
     \centering
     \includegraphics[width=.9\linewidth,height=4.0cm]{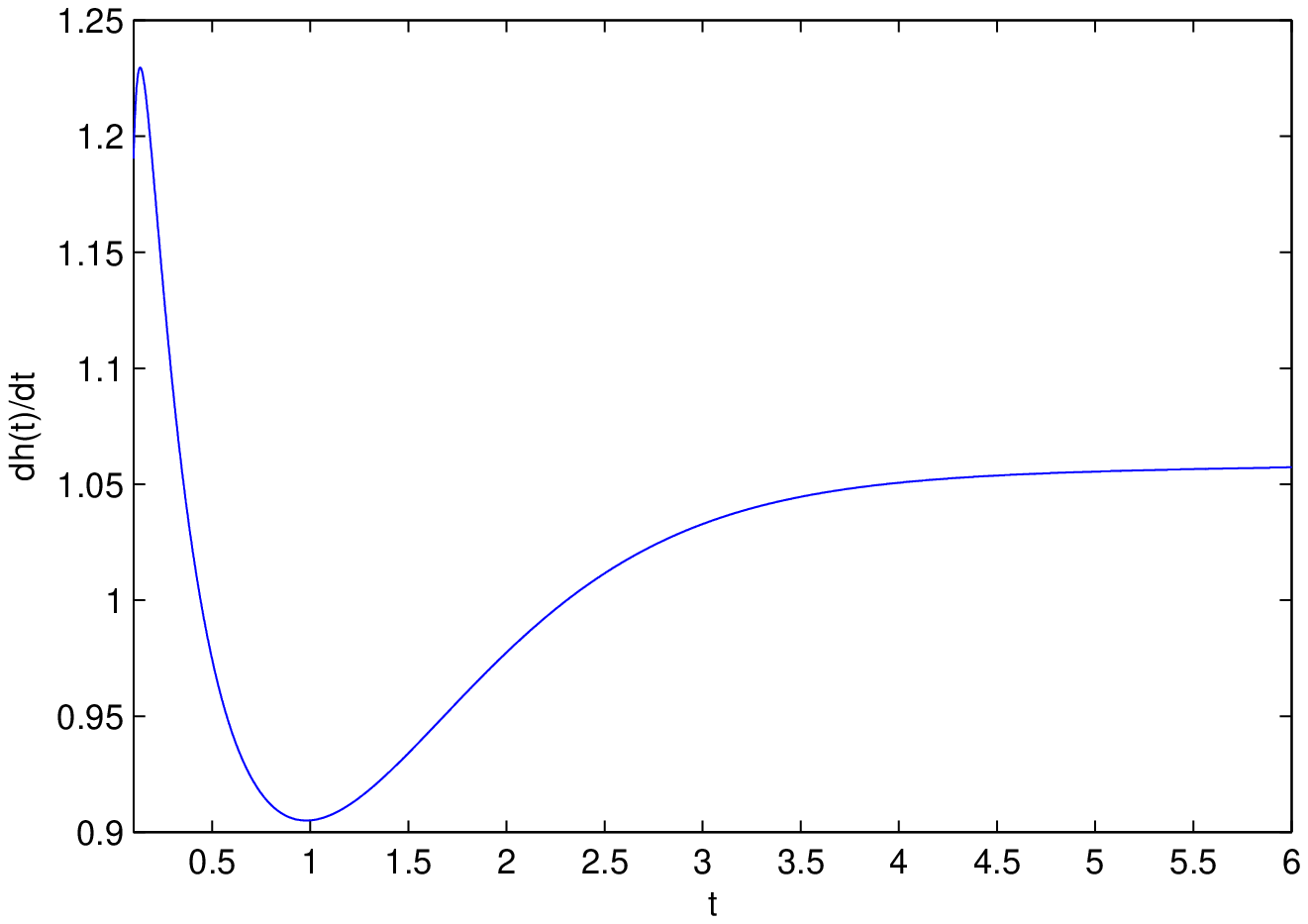}
     \caption{Evolution of the speed $\frac{h(t)}{t}$}\label{Fig:8}
   \end{minipage}
\end{figure}
\end{Ex}

\begin{Ex}\label{Exm-5}
We only change the moving speed to a smaller positive number $\nu = 0.01$ in Exmple \ref{Exm-4}. Figure \ref{Fig:9} and \ref{Fig:10} show the system \eqref{one-free-boundary-eq} has a tendency of vanishing.
\begin{figure}[!htb]
   \begin{minipage}{0.48\textwidth}
     \centering
     \includegraphics[width=.9\linewidth,height=4.0cm]{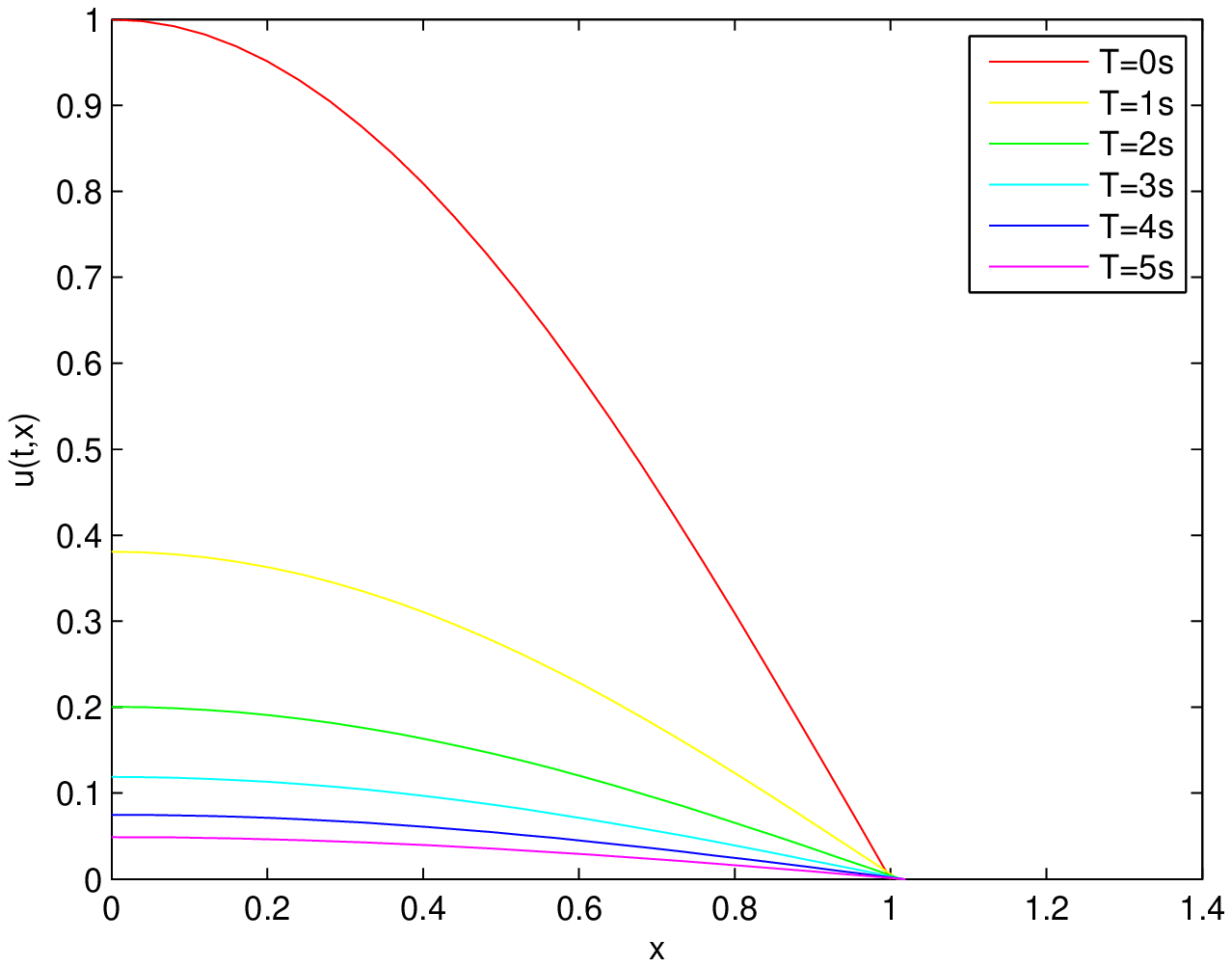}
     \caption{Evolution of the density $u(t,x)$}\label{Fig:9}
   \end{minipage}\hfill
   \begin {minipage}{0.48\textwidth}
     \centering
     \includegraphics[width=.9\linewidth,height=4.0cm]{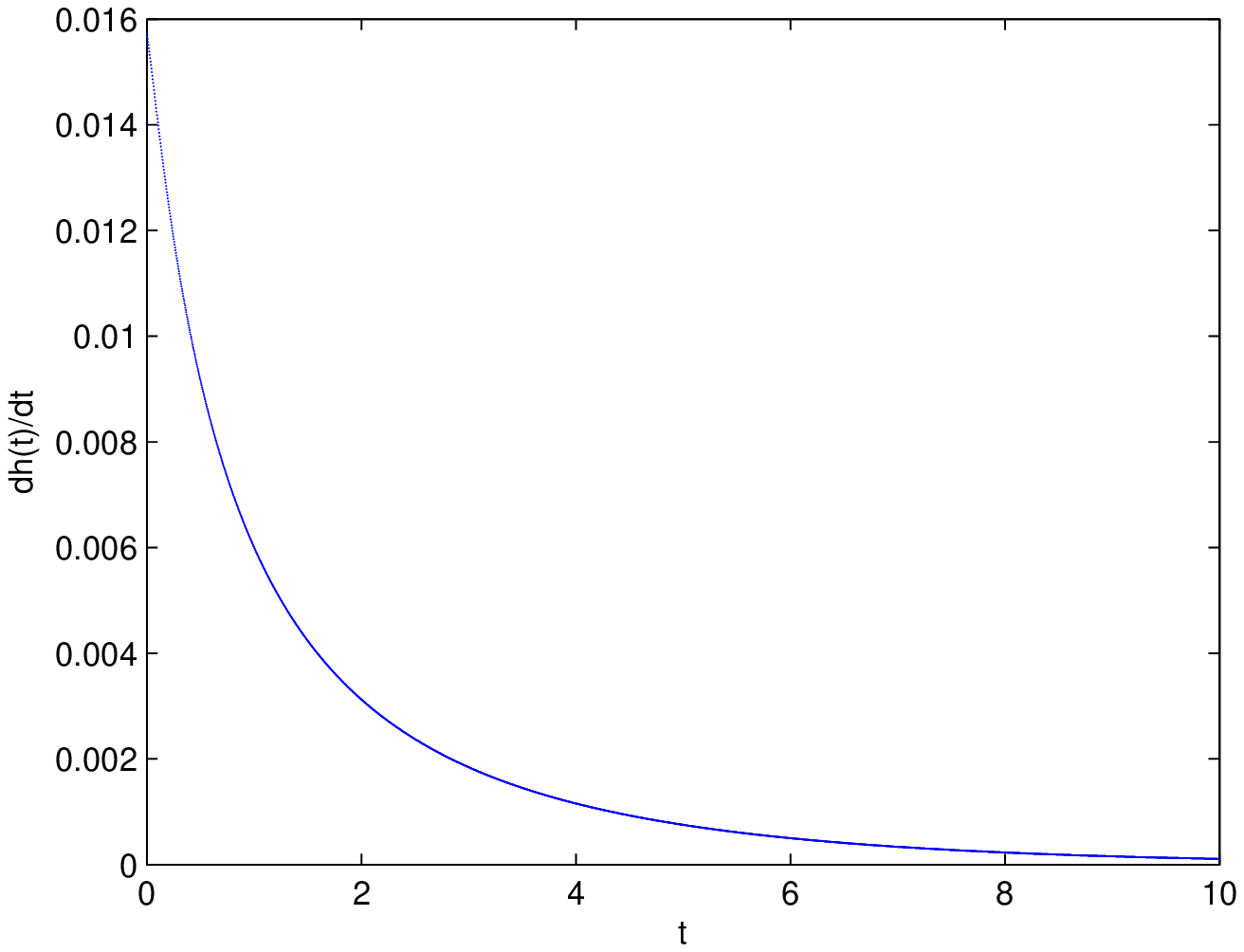}
     \caption{Evolution of the speed $\frac{h(t)}{t}$}\label{Fig:10}
   \end{minipage}
\end{figure}
\end{Ex}

\begin{Ex}\label{Exm-6}
In the logistic chemotaxis model \eqref{one-free-boundary-eq}, let $h_0 = 1.0 < l^* =1.11, u_0 = \cos(\pi x/2h_0)$ and
$(\chi_1, \chi_2, \lambda_1,\lambda_2,\mu_1, \mu_2) = (0.2, 0.1, 1, 2, 1, 2)$.
By the dichotomy method, the simulations indicate the critical $\nu^*$ is between $0.05$ and $0.025$ (see Figure \ref{Fig:11} and \ref{Fig:12}).
\begin{figure}[!htb]
   \begin{minipage}{0.48\textwidth}
     \centering
     \includegraphics[width=.9\linewidth,height=4.0cm]{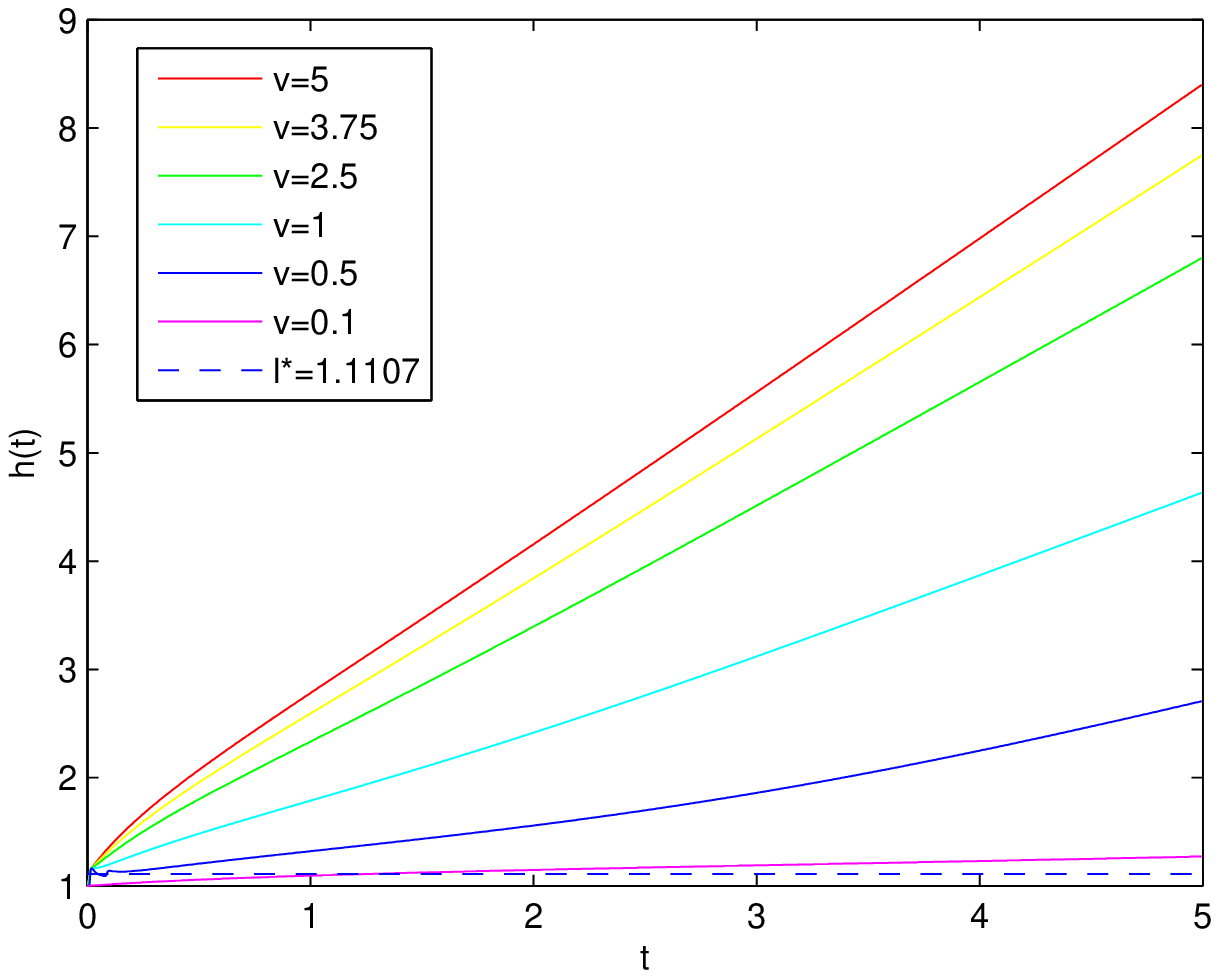}
     \caption{Evolution of the habitat length $h(t)$}\label{Fig:11}
   \end{minipage}\hfill
   \begin {minipage}{0.48\textwidth}
     \centering
     \includegraphics[width=.9\linewidth,height=4.0cm]{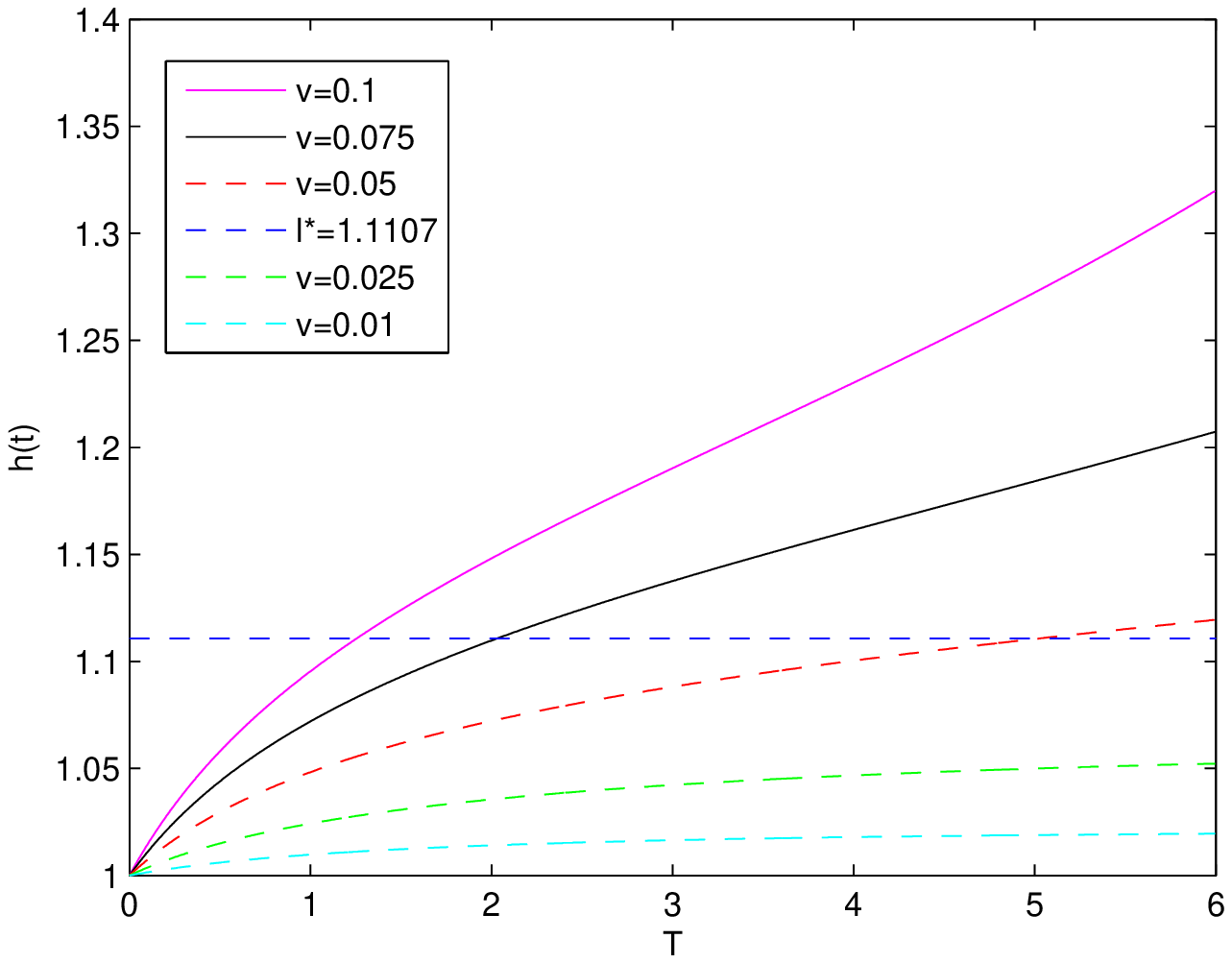}
     \caption{Evolution of the habitat length $h(t)$}\label{Fig:12}
   \end{minipage}
\end{figure}
 \end{Ex}
\subsection{Vanishing-spreading dichotomy dependence on $u_0$}
In the following, we study the vanishing-spreading dependence on initial solution $u_0(x)$.
\begin{Ex}\label{Exm-7}
We first investigate the system \eqref{one-free-boundary-eq} with large initial solution $u_0(x)=4\cos(\pi x/2h_0)$ with the following parameters $(h_0, \chi_1, \chi_2, \nu, \lambda_1,\lambda_2,\mu_1, \mu_2) = (1, 0.2, 0.1, 0.8, 1, 2, 1, 2)$.
The system has a tendency of spreading and converges to the constant $\frac{a}{b} =2$ (see Figure \ref{Fig:13},\ref{Fig:14}).
\begin{figure}[!htb]
   \begin{minipage}{0.48\textwidth}
     \centering
     \includegraphics[width=.9\linewidth,height=4.0cm]{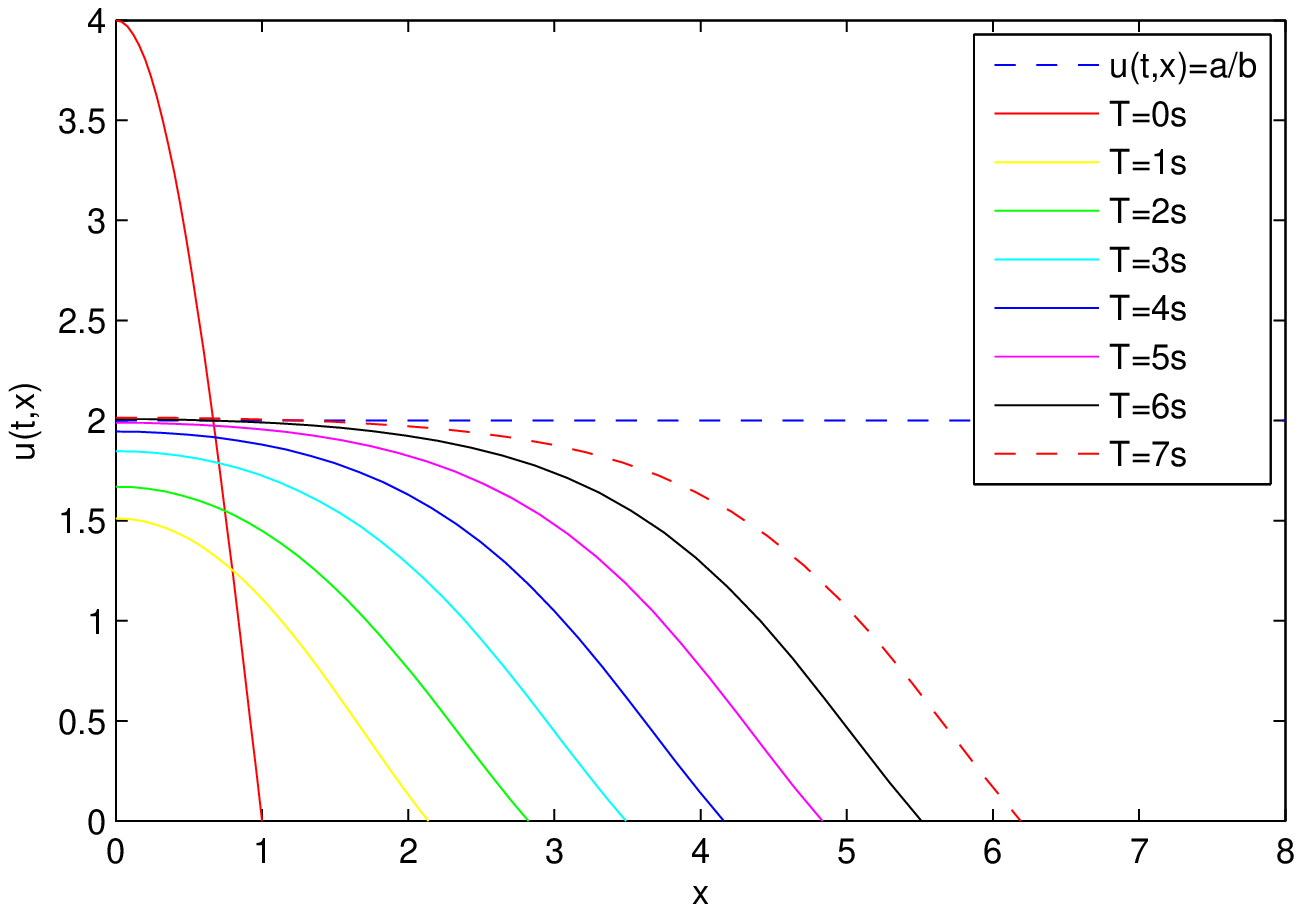}
     \caption{Evolution of the density $u(t,x)$}\label{Fig:13}
   \end{minipage}\hfill
   \begin {minipage}{0.48\textwidth}
     \centering
     \includegraphics[width=.9\linewidth,height=4.0cm]{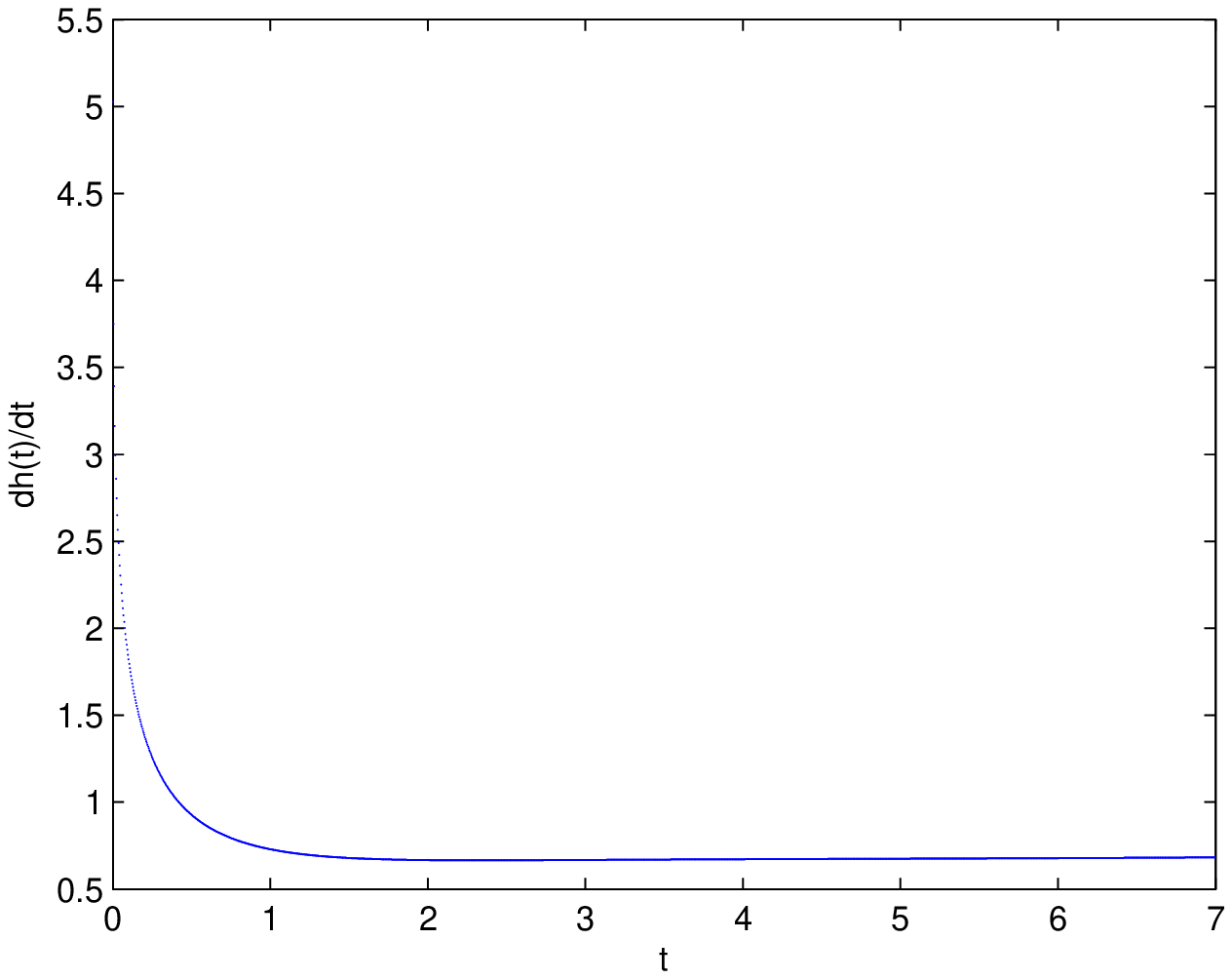}
     \caption{Evolution of the speed $\frac{h(t)}{t}$}\label{Fig:14}
   \end{minipage}
\end{figure}
\end{Ex}
\begin{Ex}\label{Exm-8}
In the case of small initial solution $u_0=0.01\cos(\pi x/2h_0)$ and with fixed other parameters as in Exmple \ref{Exm-7}, the system has a tendency of vanishing (see Figure \ref{Fig:15}, \ref{Fig:16}).
\begin{figure}[!htb]
   \begin{minipage}{0.48\textwidth}
     \centering
     \includegraphics[width=.9\linewidth,height=4.0cm]{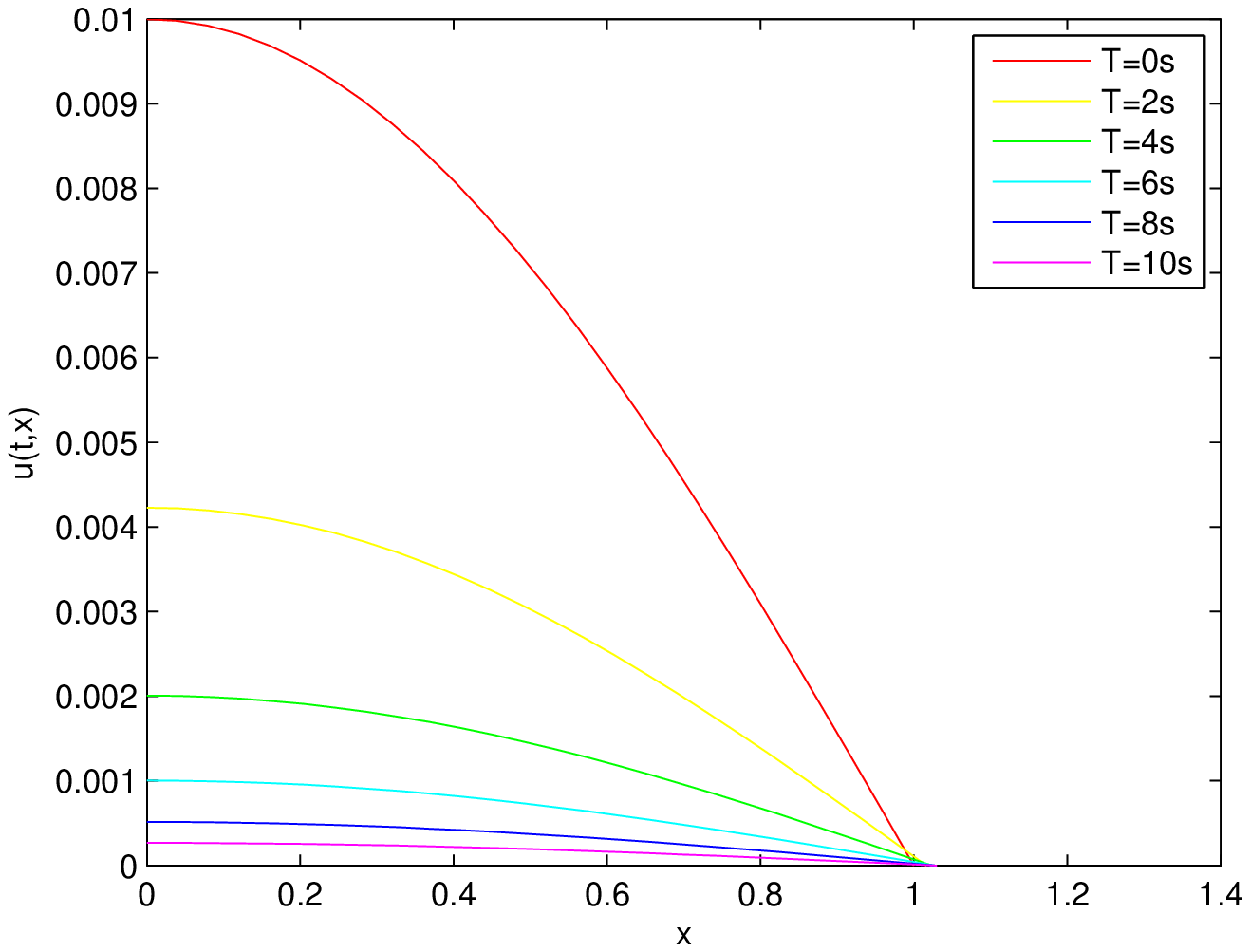}
     \caption{Evolution of the density $u(t,x)$}\label{Fig:15}
   \end{minipage}\hfill
  \begin {minipage}{0.48\textwidth}
    \centering
    \includegraphics[width=.9\linewidth,height=4.0cm]{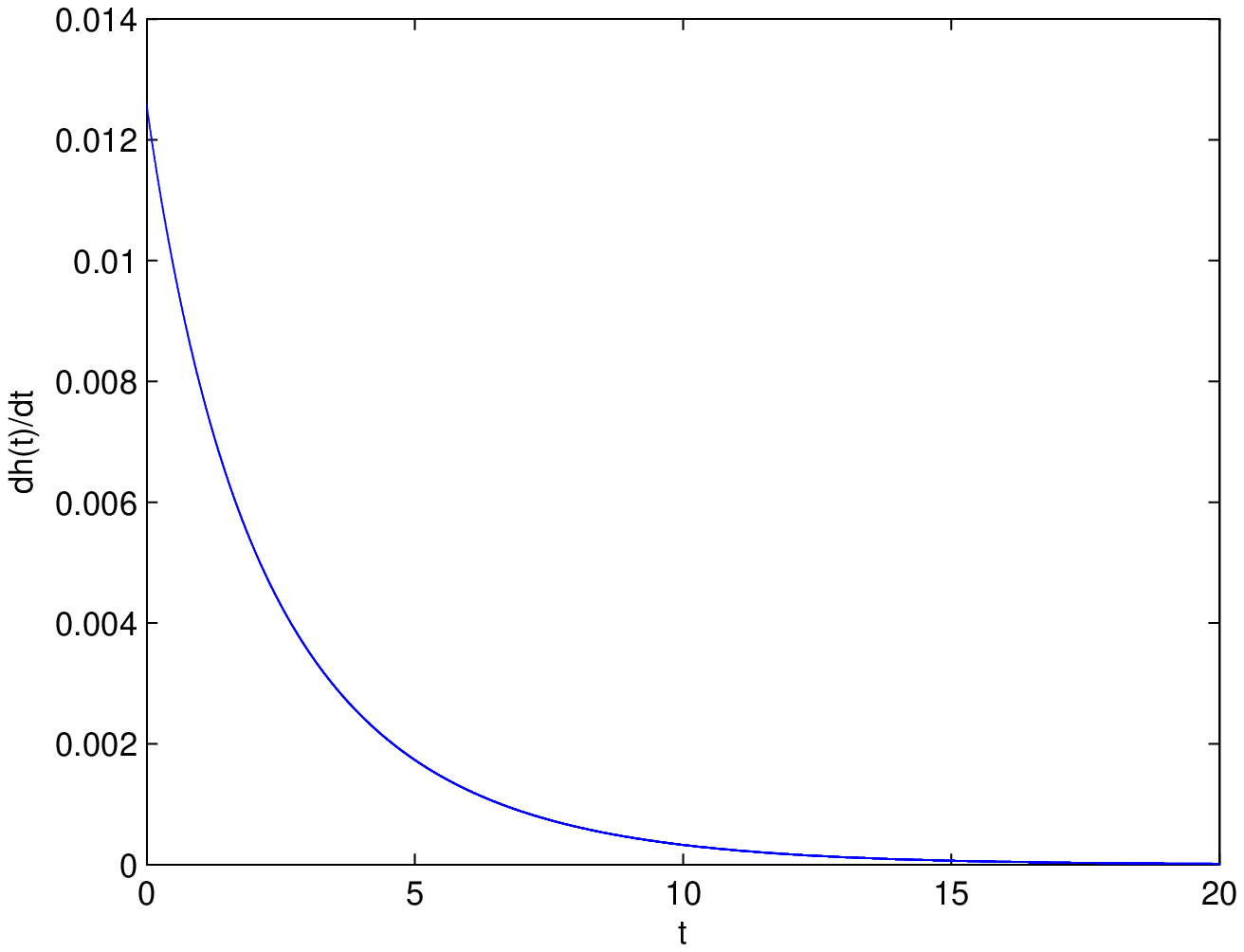}
    \caption{Evolution of the speed $\frac{h(t)}{t}$}\label{Fig:16}
   \end{minipage}
\end{figure}
\end{Ex}
In summary, the system \eqref{one-free-boundary-eq} has tendency of vanishing when the initial solution $u_0(x)$ is small and the moving boundary converges to a constant less than $l^*$. In the case of large initial solution, the system has tendency of spreading whose spreading speed also converges to a constant which is similar to the \emph{Fisher-KPP} free boundary problems.
\subsection{The dependence of asymptotic spreading speed $\frac{h(t)}{t}$ on parameters}
In the following, we mainly focus on the spreading speed's dependence on parameters $u_0, \nu$, and the chemotactic sensitivity coefficients $\chi_1, \chi_2$ when spreading happens.

We first consider the dependence of the spreading speed on the moving speed $\nu$ with small $u_0$ and $h_0>l^*$.

\begin{Ex}\label{Exm-9}
With a large $\nu$, let the parameters in the system \eqref{one-free-boundary-eq} be the following
\newline
$(h_0, u_0, \chi_1, \chi_2, \nu, \lambda_1,\lambda_2,\mu_1.\mu_2) = (2, \cos(\pi x/2h_0), 0.2, 0.1, 2, 1, 2, 1, 2)$, the system has tendency of spreading to the half line $\mathbb{R}^+$ (see Figure \ref{Fig:17}).
\begin{figure}[!htb]
   \begin{minipage}{0.48\textwidth}
     \centering
     \includegraphics[width=.9\linewidth,height=4.0cm]{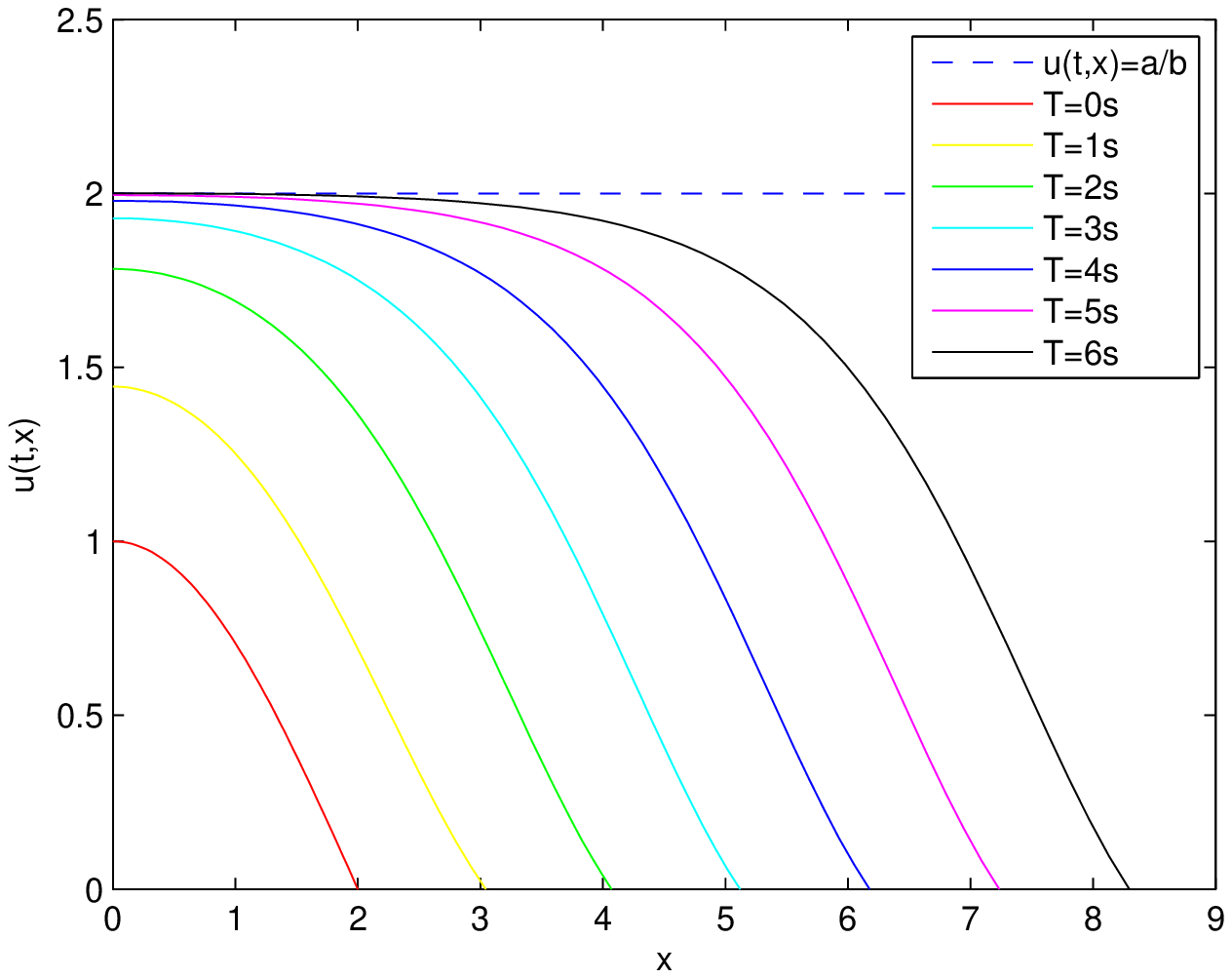}
     \caption{Evolution of the density $u(t,x)$}\label{Fig:17}
   \end{minipage}\hfill
   \begin {minipage}{0.48\textwidth}
     \centering
     \includegraphics[width=.9\linewidth,height=4.0cm]{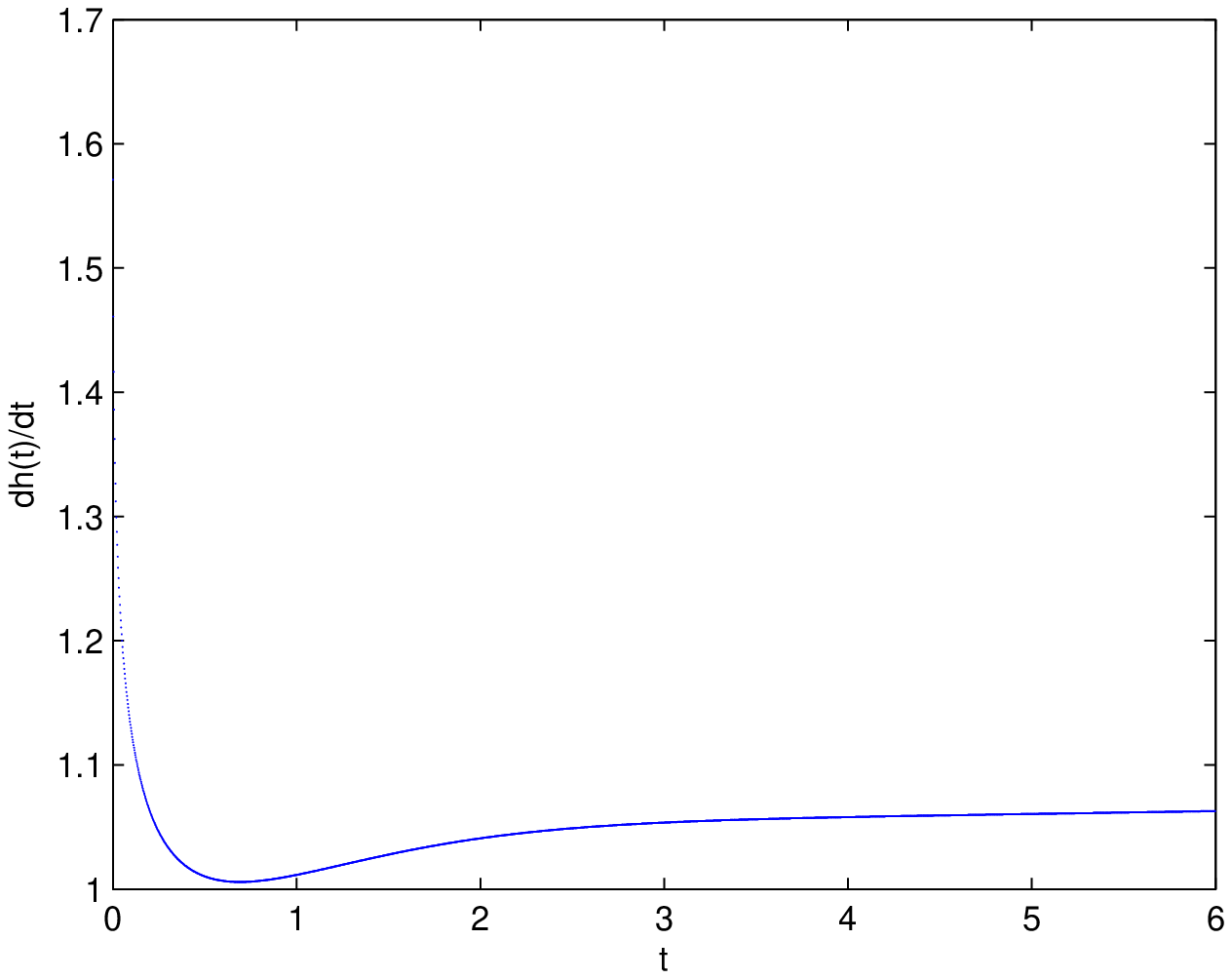}
     \caption{Evolution of the speed $\frac{h(t)}{t}$}\label{Fig:18}
   \end{minipage}
\end{figure}
\end{Ex}

\begin{Ex}\label{Exm-10}
With a even smaller moving speed $\nu=0.01$, let other parameters are the same as in Exmple \ref{Exm-9}. The system \eqref{one-free-boundary-eq} has tendency of spreading and the spreading speed is smaller compared to the system with larger $\nu$ (see Figure \ref{Fig:19}, \ref{Fig:20}).
\begin{figure}[!htb]
   \begin{minipage}{0.48\textwidth}
     \centering
     \includegraphics[width=.9\linewidth,height=4.0cm]{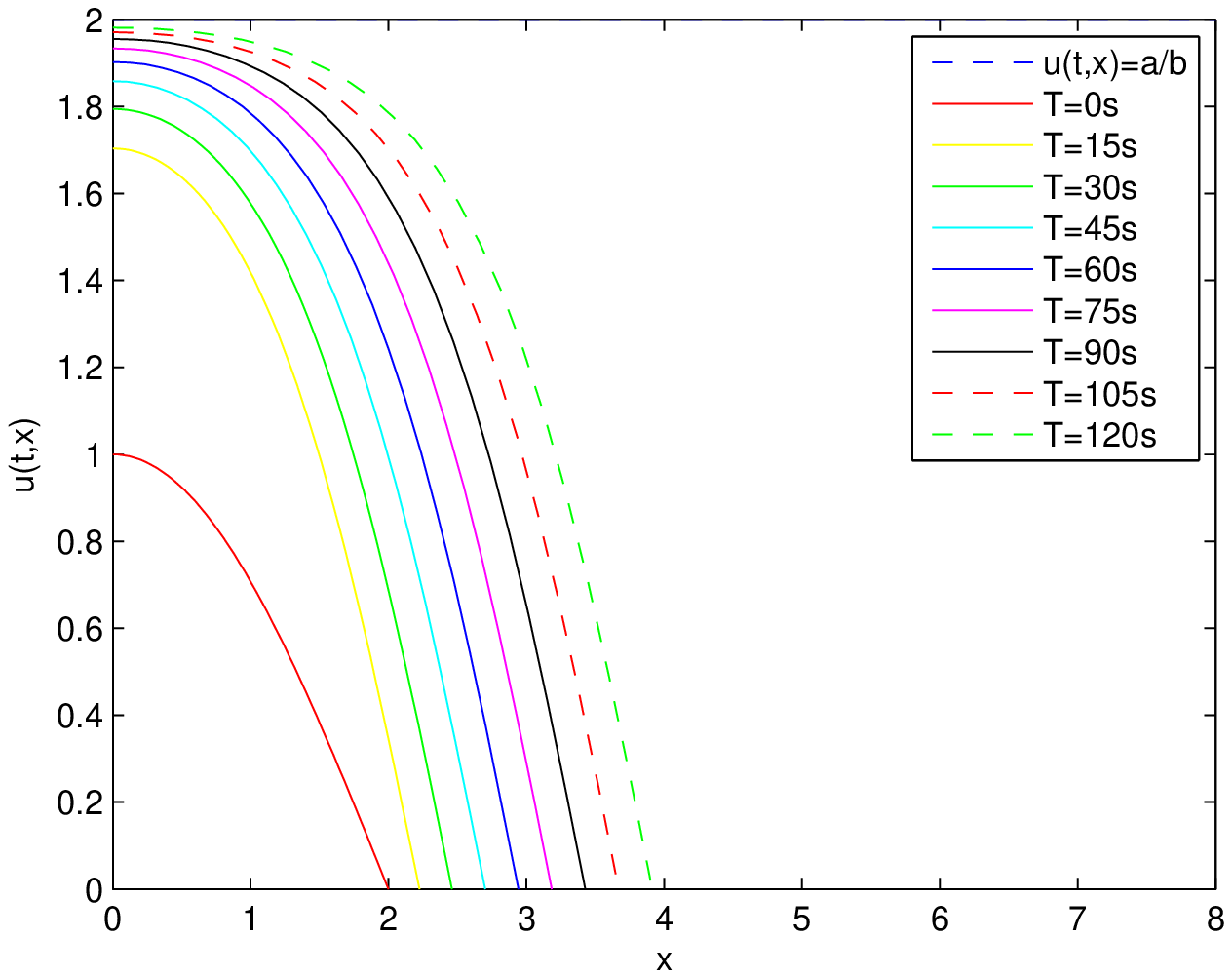}
     \caption{Evolution of the density $u(t,x)$}\label{Fig:19}
   \end{minipage}\hfill
   \begin {minipage}{0.48\textwidth}
     \centering
     \includegraphics[width=.9\linewidth,height=4.0cm]{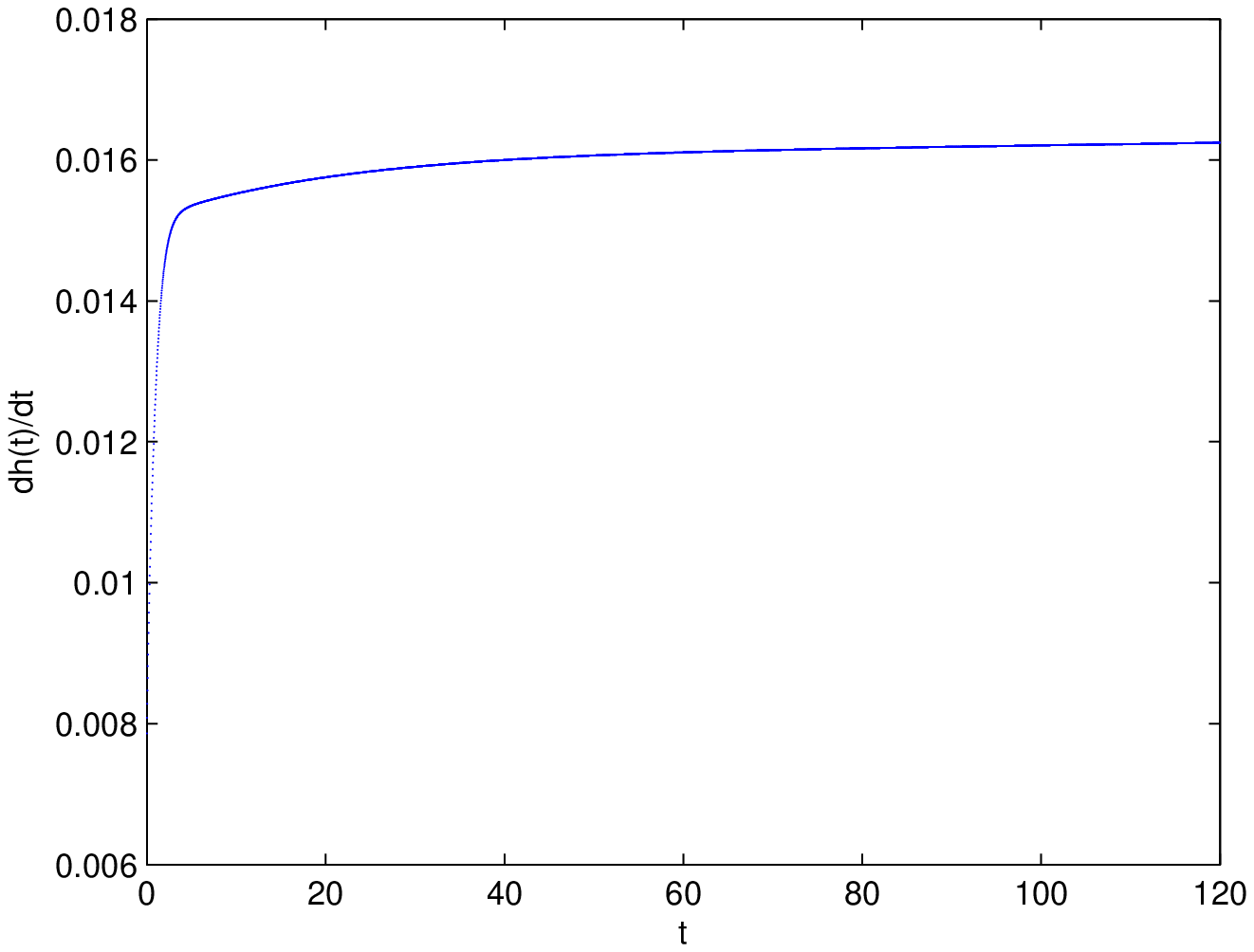}
     \caption{Evolution of the speed $\frac{h(t)}{t}$}\label{Fig:20}
  \end{minipage}
\end{figure}
\end{Ex}
The above simulations indicate that the spreading happens when the moving speed $\nu$ is small and we also can conclude that the asymptotic spreading speed depends on the moving speed $\nu$. Figure \ref{Fig:11} and \ref{Fig:12} also indicate that when the spreading happens, the spreading speed is a increasing function of moving speed $\nu$.

In the following, we use a table to compare the spreading speeds with different initial solution $u_0$, initial habitat $h_0$, and the effects of the chemotactic sensitivity $\chi_1, \chi_2$. These simulations indicate the asymptotic spreading speed is independent of the parameters $u_0, h_0, \chi_1, \chi_2$.
\begin{Ex}\label{Exm-11}
Let $(h_0, \chi_1, \chi_2, \nu, \lambda_1,\lambda_2,\mu_1.\mu_2) = (2, 0.2, 0.1, 0.8, 1, 2, 1, 2)$
and $u_0=\sigma\cos(\pi x/2h_0).$ With different $\sigma$ choices, we have the following spreading speed data.
\begin{center}
\begin{tabular}{|c|c|c|c|c|c|}
  \hline
  dh(t)/dt&$\sigma=0.01$&$\sigma=0.1$&$\sigma=1$&$\sigma=2$&$\sigma=4$\\
  \hline
  T=1s&     0.024    &  0.192  & 0.607 & 0.703 & 0.750\\
  \hline
  T=2s&     0.087    &  0.417  & 0.659 & 0.678 & 0.678\\
  \hline
  T=3s&    0.255    &  0.595  & 0.676 & 0.675 & 0.667\\
  \hline
  T=4s&     0.483    &  0.664  & 0.680 & 0.676 & 0.666\\
  \hline
  T=5s&     0.626    &  0.682  & 0.681 & 0.677 & 0.666\\
  \hline
  T=6s&     0.673    &  0.687  & 0.682 & 0.677 & 0.668\\
  \hline
  T=7s&     0.685    &  0.689  & 0.683 & 0.678 & 0.668\\
  \hline
  T=8s&     0.689    &  0.690  & 0.683 & 0.679 & 0.670\\
  \hline
  T=9s&     0.690    &  0.691  & 0.684 & 0.680 & 0.672\\
  \hline
  T=10s&    0.691    &  0.692  & 0.684 & 0.681 & 0.673\\
  \hline
  \hline
\end{tabular}
\end{center}
The table indicate the spreading speed converges to a constant near $0.7$ with different choices of $\sigma$. The smaller $\sigma$ has a slower spreading speed at the beginning, but converges to a similar constant to other larger choice of $\sigma$ as time increase. For the larger $\sigma$, the initial spreading speed decrease first and then converges to a constant which is independent of the initial $u_0$.
\end{Ex}

In \emph{Fisher-KPP} free boundary problem, the spreading speed is also independent of the initial habitat when spreading happens. Because of the lack of comparison principle in chemotaxis system, this result is still open. Our simulation indicate this results should be also exist in chemotaxis system.
\begin{Ex}\label{Exm-12}
Let the parameters in the system \eqref{one-free-boundary-eq} as $(h_0, \chi_1, \chi_2, \nu, \lambda_1,\lambda_2,\mu_1.\mu_2) = (2, 0, 0,\\ 0.8, 1, 2, 1, 2)$ and $u_0=\sigma\cos(\pi x/2h_0)$. we have the following spreading speed data.
\begin{center}
\begin{tabular}{|c|c|c|c|c|c|}
  \hline
  dh(t)/dt&$\sigma=0.01$&$\sigma=0.1$&$\sigma=1$&$\sigma=2$&$\sigma=4$\\
  \hline
  T=1s&     0.024    &  0.191  & 0.606 & 0.710 & 0.766\\
  \hline
  T=2s&     0.087    &  0.416  & 0.662 & 0.687 & 0.698\\
  \hline
  T=3s&    0.254    &  0.594  & 0.680 & 0.686 & 0.688\\
  \hline
  T=4s&     0.486    &  0.664  & 0.685 & 0.687 & 0.687\\
  \hline
  T=5s&     0.627    &  0.682  & 0.687 & 0.687 & 0.6875\\
  \hline
  T=6s&     0.673    &  0.6865  & 0.688 & 0.688 & 0.688\\
  \hline
  T=7s&     0.684    &  0.688  & 0.688 & 0.688 & 0.688\\
  \hline
  T=8s&     0.687    &  0.6885  & 0.689 & 0.688 & 0.6885\\
  \hline
  T=9s&     0.688    &  0.689  & 0.689 & 0.689 & 0.689\\
  \hline
  T=10s&    0.689    &  0.690  & 0.690 & 0.689 & 0.689\\
  \hline
  \hline
\end{tabular}
\end{center}
\end{Ex}

\begin{Ex}\label{Exm-13}
Let the parameters in the system \eqref{one-free-boundary-eq} as $u_0 = \cos(\pi x/2h_0)$ and
\newline
 $(\chi_1, \chi_2, \nu, \lambda_1,\lambda_2,\mu_1.\mu_2)$ = $(0.2, 0.1, 0.8, 1, 2, 1, 2)$.
With different $h_0$ choices, we have the following spreading speeds table:
\begin{center}
\begin{tabular}{|c|c|c|c|c|c|c|}
  \hline
  dh(t)/dt &$h_0=1$  & $h_0=1.2$   & $h_0=1.5$   & $h_0=2$   & $h_0=3$ &$h_0=5$ \\
  \hline
  T=1s     &  0.416  & 0.559   & 0.602   & 0.618             & 0.697  &0.538 \\
  \hline
  T=2s     & 0.492     & 0.607  &0.642   &0.661              & 0.663  &0.652 \\
  \hline
  T=3s     & 0.596     & 0.654  &0.669   &0.676              & 0.678  &0.677 \\
  \hline
  T=4s     & 0.654     & 0.674  &0.679   &0.680              & 0.681  &0.682 \\
  \hline
  T=5s     & 0.674     & 0.680  &0.682   &0.681              & 0.682  &0.683 \\
  \hline
  T=6s     & 0.680   & 0.682  &0.683   &0.682               & 0.682  &0.683 \\
  \hline
  T=7s     & 0.682  &0.6835   &0.684   &0.682                & 0.683  &0.684 \\
  \hline
  T=8s     & 0.683  &0.6845   &0.685   &0.683                & 0.683  &0.684 \\
  \hline
  T=9s     & 0.6835 & 0.685  &0.686   &0.6835                  & 0.6835  &0.685\\
  \hline
  T=10s    & 0.684  & 0.686  &0.687   &0.684                 & 0.684  &0.685 \\
  \hline
\end{tabular}
\end{center}
\end{Ex}
\emph{Fisher-KPP} free boundary problem \cite{DuLi} is a special case of the chemotaxis free boundary problem with $\chi_1 = 0, \chi_2 = 0$, which are fully investigated and its asymptotic spreading speed $\frac{h(t)}{t}$ is only depending on $a(t,x)$. Our numerical simulations indicate the speed may be independent of the chemotactic sensitivity $\chi_1, \chi_2$.

\begin{Ex}\label{Exm-14}
Let the parameters in the system \eqref{one-free-boundary-eq} as $u_0 = \cos(\pi x/2h_0)$ and
\newline
 $(\chi_1, \chi_2, \nu, \lambda_1,\lambda_2,\mu_1.\mu_2) = (0, 0, 0.8, 1, 2, 1, 2)$. With different $h_0$ choices, we have the following spreading speeds table:
\begin{center}
\begin{tabular}{|c|c|c|c|c|c|c|}
  \hline
  dh(t)/dt &$h_0=1$  & $h_0=1.2$   & $h_0=1.5$   & $h_0=2$   & $h_0=3$ &$h_0=5$ \\
  \hline
  T=1s     &  0.416  & 0.560   & 0.572   & 0.620             & 0.601  &0.538 \\
  \hline
  T=2s     & 0.493     & 0.609  &0.632   &0.665              & 0.669  &0.658 \\
  \hline
  T=3s     & 0.599     & 0.658  &0.668   &0.680              & 0.683  &0.683 \\
  \hline
  T=4s     & 0.658     & 0.679  &0.682   &0.685              & 0.686  &0.687 \\
  \hline
  T=5s     & 0.679     & 0.685  &0.686   &0.687              & 0.687  &0.688 \\
  \hline
  T=6s     & 0.6855   & 0.687  &0.687   &0.687               & 0.6875  &0.688 \\
  \hline
  T=7s     & 0.687  &0.688   &0.687   &0.6875                & 0.688  &0.689 \\
  \hline
  T=8s     & 0.688  &0.689   &0.688   &0.688                & 0.688  &0.689 \\
  \hline
  T=9s     & 0.689 & 0.689  &0.688   &0.688                  & 0.6885  &0.690\\
  \hline
  T=10s    & 0.689  & 0.690  &0.688   &0.688                & 0.689  &0.690 \\
  \hline
\end{tabular}
\end{center}
\end{Ex}
\subsection{Asymptotic behaviors with large chemotactic sensitivities}
The long time behaviors of the system is depend on the choices of parameters \cite{Tello2007chemotaxis}. In order to guarantee the convergence of the system to the constant $a/b =2$, we need chemotactic sensitivity coefficient to be small enough, and if not, the system may converges to some other constants. The following simulations indicate such result.
\begin{Ex}\label{Exm-15}
Let $h_0 = 2.5 > l^* =1.11$, $u_0=\cos(\pi x/2h_0)$, and
$(\chi_1, \chi_2, \nu, \lambda_1,\lambda_2,\mu_1.\mu_2) = (0.2, 0.1, 0.8, 2, 1, 2, 1)$.
Simulation indicates the system has spreading tendency but does not converge to any constant (see Figure \ref{Fig:21}).
\begin{figure}[!htb]
  \centering
  \includegraphics[height=5cm,width=10cm]{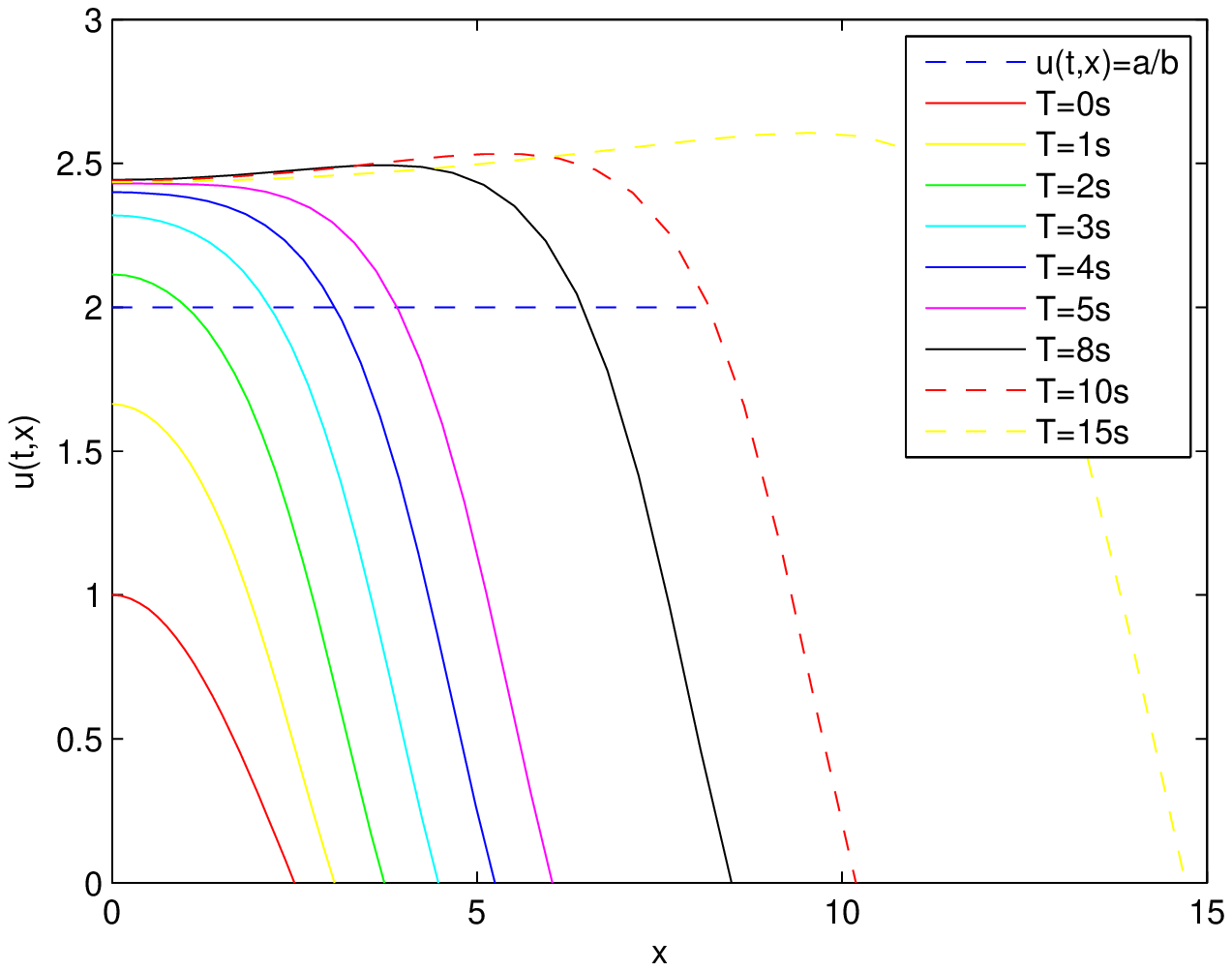}\\
 \caption{Evolution of the density $u(t,x)$}\label{Fig:21}
\end{figure}
\end{Ex}

\begin{Ex}\label{Exm-16}
Compared to Exmple \ref{Exm-15}, fix other parameters and let $(\lambda_1,\lambda_2,\mu_1,\mu_2)=(2,1,1,2)$, the simulation indicate the system converge to the constant $a/b = 2$ (see Figure \ref{Fig:22}).
\begin{figure}[!htb]
   \begin{minipage}{0.48\textwidth}
     \centering
     \includegraphics[width=.9\linewidth,height=4.0cm]{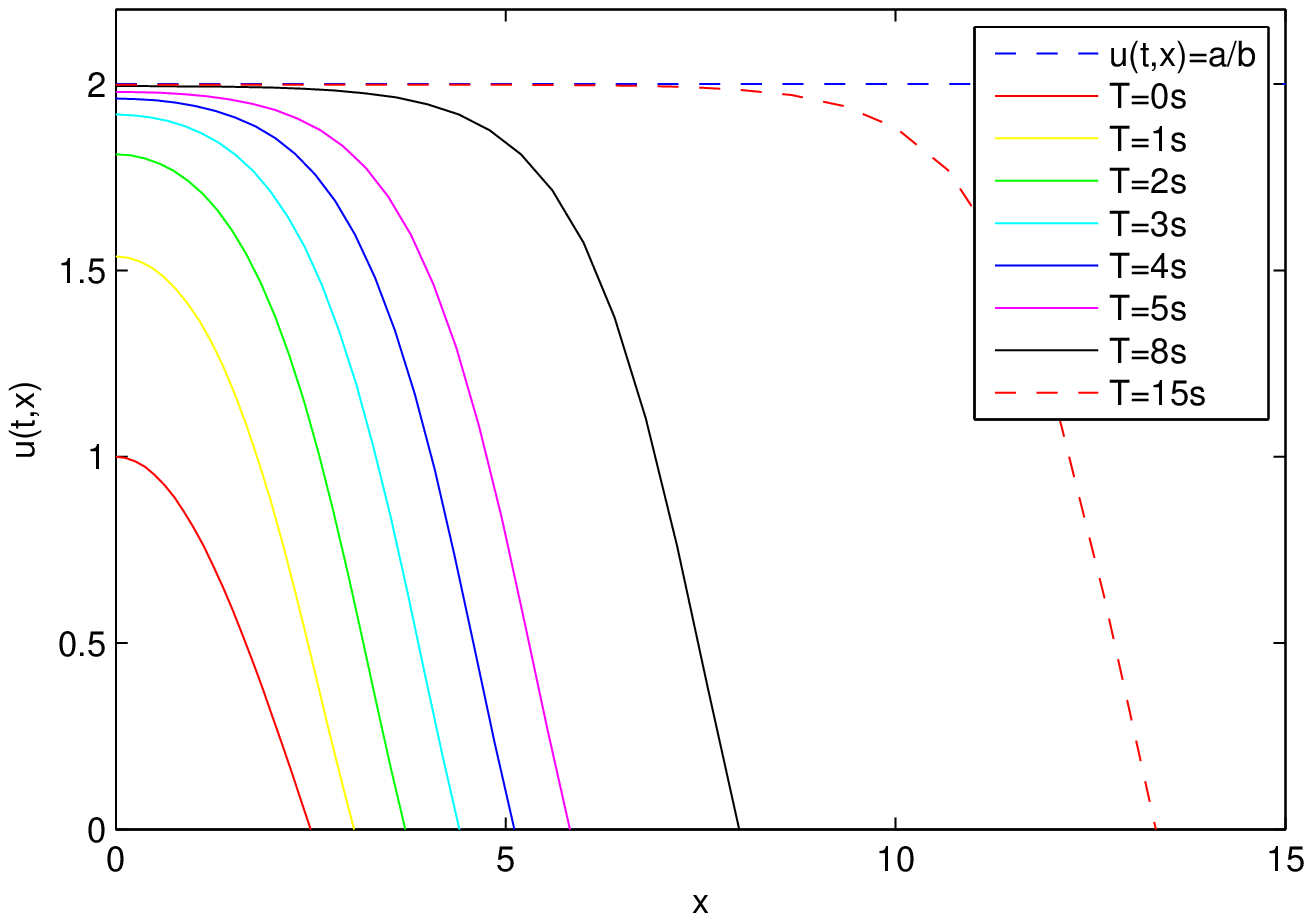}
     \caption{Evolution of the density $u(t,x)$}\label{Fig:22}
   \end{minipage}\hfill
   \begin {minipage}{0.48\textwidth}
     \centering
     \includegraphics[width=.9\linewidth,height=4.0cm]{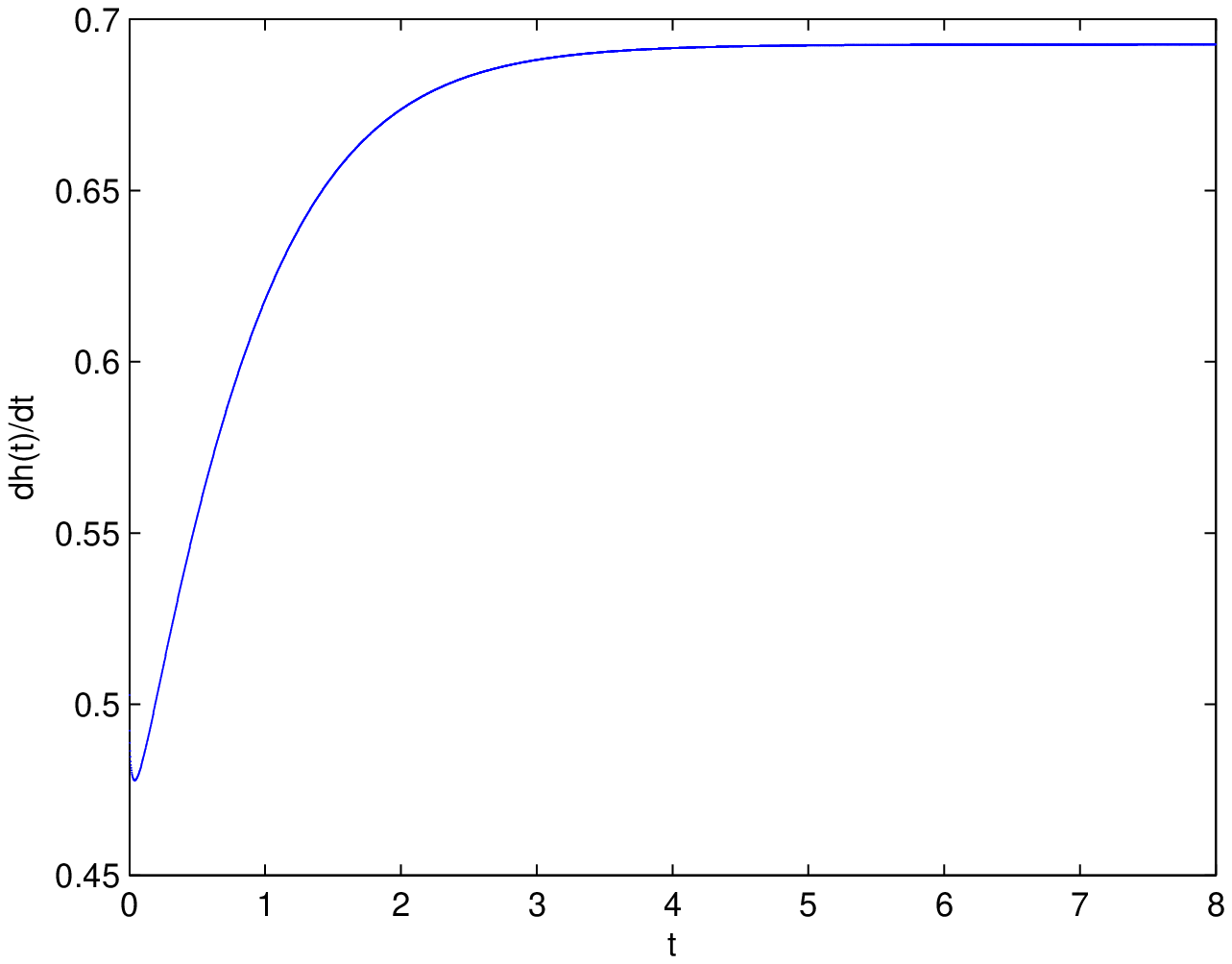}
     \caption{Evolution of the speed $\frac{h(t)}{t}$}\label{Fig:23}
   \end{minipage}
\end{figure}

The corresponding spreading speed $\frac{h(t)}{t}$ tends to a positive constant 0.69 (see Figure \ref{Fig:23}),which is similar to \emph{Fisher-KPP} free boundary problems. It is an evidence that the spreading speed is independent of the ratio of $a/\lambda_1$ \cite{SaShXu}.
\end{Ex}

\begin{Ex}\label{Exm-16}
Compared to Exmple \ref{Exm-15}, fix other parameters and let $\lambda_1<\lambda_2$ such that $(\lambda_1,\lambda_2,\mu_1,\mu_2)=(1,2,2,1)$, we have the following spreading result which does not converge to $a/b = 2$.
\begin{figure}[!htb]
  \centering
  \includegraphics[height=5cm,width=10cm]{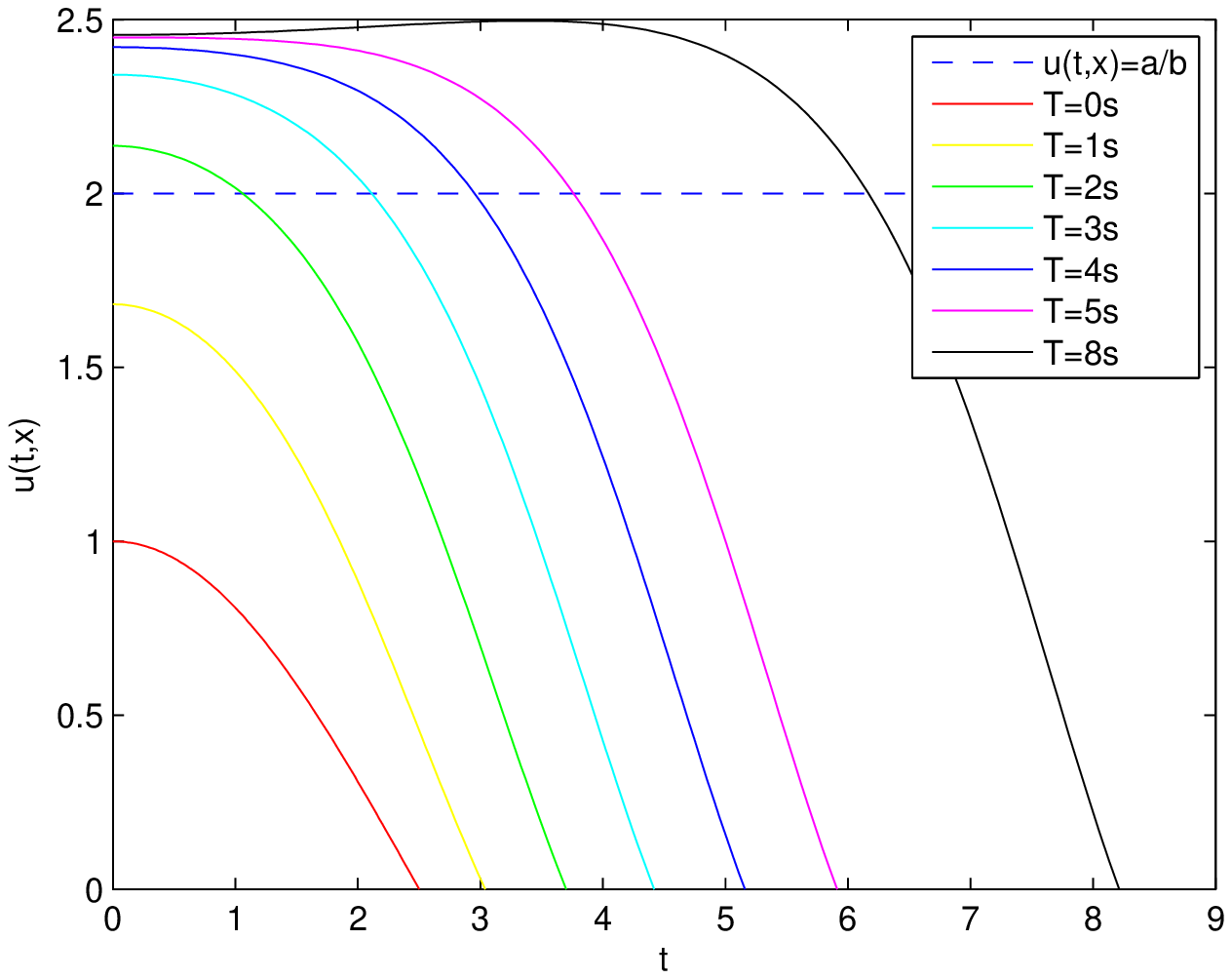}\\
 \caption{Evolution of the density $u(t,x)$}\label{4}
\end{figure}
\end{Ex}
\section{Conclusion}
In this paper we construct a numerical scheme to approximate the continuous logistic type chemotaxis system with a free boundary. The scheme's error estimate, positivity preserving, the monotonicity of the free boundary and stability are investigated. Numerical simulations validate some proved theoretical results such as vanishing-spreading dichotomy, persistency and stability. Compared to \emph{Fisher-KPP} free boundary problem, the dependence of the vanishing-spreading dichotomy on initial solution $u_0$ and initial habitat $h_0$ are still open problems both theoretically and numerically. Furthermore, the existence of the asymptotic spreading speed $\frac{h(t)}{t}$ and its dependence on the parameters are also open problems. All these continuous and discrete dynamical questions should be investigated in the future.

\section*{Acknowledgement}
The authors appreciate Prof. Wenxian Shen at Auburn University for introducing this interesting topic and many valuable and insightful discussions and help.


\begin{thebibliography}{99}



 \bibitem{BaoShen1} L. Bao and W. Shen, Logistic type attraction-repulsion chemotaxis systems with a free boundary or unbounded boundary.
I. Asymptotic dynamics in fixed unbounded domain, preprint.

 \bibitem{BaoShen2} L. Bao and W. Shen, Logistic type attraction-repulsion chemotaxis systems with a free boundary or unbounded boundary.
II. Spreading-vanishing dichotomy in a domain with a free boundary, preprint.

\bibitem{Bellomo2015Toward}
N.~Bellomo, A.~Bellouquid, Y.~Tao, and M.~Winkler,
\newblock {Toward a mathematical theory of Keller-Segel models of pattern
  formation in biological tissues}.
\newblock {\em Math. Models Methods Appl. Sci.}, (25):1663--1763, 2015.
%


\bibitem{Chiu2007optimal}
C.~C. Chiu and J.~L. Yu,
\newblock {An optimal adaptive time-stepping scheme for solving reaction-diffusion-chemotaxis systems}.
\newblock {\em Math. Biosci. Eng.}, 4(2):187--203, 2007.


%
\bibitem{Diaz1995Symmtr}
J.~I. Diaz and T.Nagai,
\newblock {Symmetrization in a parabolic-elliptic system related to
  chemotaxis}.
\newblock {\em Advances in Mathematical Science and Applications},
  (5):659--680, 1995.
%
\bibitem{Diaz1998Symmtr}
J.~I. Diaz, T.Nagai, and J.-M Rakotoson,
\newblock {Symmetrization techniques on unbounded domains: Application to a
  chemotaxis system on $\mathbb{R}^N$ }.
\newblock {\em J. Differential Equations}, (145):156--183, 1998.


\bibitem{DuLi} Y.-H. Du and Z.-G. Lin,
\newblock Spreading-vanishing dichotomy in the diffusive logistic model with a free
boundary.
\newblock{\em SIAM J. Math. Anal.}, 42, 377-405, 2010.




\bibitem{du2013pulsating}
Y.-H. Du and X. Liang,
\newblock Pulsating semi-waves in periodic media and spreading speed determined by a free boundary model.
\newblock {\em Ann. Inst. H. Poincar$\acute{e}$ Anal. Non Lin$\acute{e}$aire}, 32(2) 279-305, 2015.




%
%

%
\bibitem{Galakhov2016on}
E.~Galakhov, O.~Salieva, and J.~I. Tello,
\newblock {On a parabolic-elliptic system with chemotaxis and logistic type
  growth}.
\newblock {\em J. Differential Equations}, 261(8):4631--4647, 2016.


%











\bibitem{Horstmann2005bound}
D.~Horstmann and M.~Winkler,
\newblock Boundedness vs. blow-up in a chemotaxis system,
\newblock {\em J. Differential Equations}, 215(1):52--107, 2005.


\bibitem{Issa-Shen-2017}
Tahir B. Issa and W.~Shen,
\newblock Dynamics in chemotaxis models of parabolic-elliptic type on bounded domain with time and space dependent logistic sources,
\newblock {\em  SIAM J. Appl. Dyn. Syst.}, 16 (2):926-973, 2017.

%

%
%
\bibitem{Kanga2016blow}
K.~Kanga and A.~Steven,
\newblock {Blowup and global solutions in a chemotaxis-growth system},
\newblock {\em Nonlinear Analysis}, (135):57--72, 2016.
%
\bibitem{Keller1970Initiation}
E. F. Keller and L. A. Segel,
\newblock {Initiation of slime mold aggregation viewed as an instability},
\newblock {\em J. Theoret. Biol.}, (26):399--415, 1970.

\bibitem{Keller1971Model}
E. F. Keller and L. A. Segel,
\newblock {A Model for chemotaxis},
\newblock {\em J. Theoret. Biol.}, (30):225--234, 1971.
%
%
%
%

\bibitem{Landau1950heat} H. G. Landau, Heat conduction in a melting solid {\it Quaterly of Applied Mathematics,} {\bf 1} (1950), no. 8, 81-94.

\bibitem{Li2016duffusive}
F. Li, X. Liang, and W. Shen,
\newblock Diffusive KPP equations with free boundaries in time almost periodic environments: I. Spreading and vanishing dichotomy, \newblock {\em Discrete Contin. Dyn. Syst},36(6), 3317-3338, 2016.

\bibitem{Li2016duffusive2}
F. Li, X. Liang, and W. Shen,
\newblock Diffusive KPP equations with free boundaries in time almost periodic environments: II. Spreading speeds and semi-wave solutions,
\newblock{\em J. Differential Equations}, 261 (4), 2403-2445, 2016.

\bibitem{Li2000generalized}
R.~H. Li, Z.~Y. Chen, and W. Wu,
\newblock Generalized difference methods for differential equations- Numerical analysis of finite volume methods.
\newblock{\em Marcel Dekker, Inc,} 2000.
\bibitem{Li2017local}
X.~J. Li, C.~W. Shu, and Y. Yang,
\newblock Local discontinuous Galerkin method for the Keller-Segel chemotaxis model,
\newblock{\em J. Sci. Comput.}, 73, 943-967, 2017.

\bibitem{Liu2018positivity}
J.~G. Liu, L. Wang, and Z.~N. Zhou,
\newblock Positivity-preserving and asymptotic preserving method for 2D Keller-Segal equations,
\newblock{\em Math. Comp.}, 87 (311), 1165-1189, 2018.

\bibitem{Liu2018numerical}
S. Liu, X.~F. Liu,
\newblock Numerical methods for a wwo-species
competition-diffusion model with free boundaries,
\newblock{\em Mathematics}, 6 (5), 72-96, 2018.

\bibitem{Liu2019numerical}
S. Liu, Y.~H. Du, and X.~F. Liu,
\newblock Numerical Studies of a Class of Reaction-Diffusion Equations with Stefan Conditions,
\newblock{\em  International Journal of Computer Mathematics}, 2019.


%
\bibitem{lockwood2007invasion}
J.~L. Lockwood, M.~F. Hoopes, and M.~P. Marchetti,
\newblock {\em {Invasion Ecology}},
\newblock Blackwell Publishing, 2007.






%


\bibitem{Piqueras2017afront} M.-A. Piqueras, R. Company, L. L$\acute{o}$dar, A front-fixing numerical method for a free boundary nonlinear diffusion logistic population model, {\it J. Comput. Appl. Math.}, 309, 473-481, 2017.

\bibitem{Nagai1997application}
T.~Nagai, T.~Senba, and K.~Yoshida,
\newblock {Application of the Trudinger-Moser Inequality to a Parabolic System
  of Chemotaxis},
\newblock {\em Funkcialaj Ekvacioj}, 40:411--433, 1997.



\bibitem{Saito2005Notes}
N.~Saito and T.~Suzuki,
\newblock {Notes on finite difference schemes to a parabolic-elliptic system modelling chemotaxis},
\newblock {\em Appl. Math. Comput.}, 171, 72-90, 2005.

\bibitem{Salako2016spreading}
R.~B. Salako and W. Shen,
\newblock {Spreading Speeds and Traveling waves of a parabolic-elliptic chemotaxis system with logistic source on $\mathbb{R}^N$},
\newblock {\em  Discrete Contin. Dyn. Syst.},  37(12), 6189--6225, 2017.





\bibitem{SaShXu}
R.~B. Salako, W. Shen, and S.~W. Xue,
\newblock{Can chemotaxis speed up or slow down the spatial spreading in parabolic-elliptic chemotaxis  systems with logistic source?} preprint.

\bibitem{Shigesada1997biological}
N.~Shigesada and K.~Kawasaki,
\newblock {\em {\textit{Biological Invasions: Theory and Practice}, Oxford
  Series in Ecology and Evolution}},
\newblock Oxford Univ. Press., Oxford, 1997.
%
\bibitem{Sugiyama2006global1}
Y.~Sugiyama,
\newblock {Global existence in sub-critical cases and finite time blow up in
  super critical cases to degenerate Keller-Segel systems},
\newblock {\em Differential Integral Equations}, 19(8):841--876, 2006.

\bibitem{Sugiyama2006global2}
Y.~Sugiyama and H.~Kunii,
\newblock {Global Existence and decay properties for a degenerate keller-Segel
  model with a power factor in drift term},
\newblock {\em J. Differential Equations}, 227:333--364, 2006.


\bibitem{Tello2007chemotaxis}
J.~I. Tello and M.~Winkler,
\newblock {A chemotaxis system with logistic source},
\newblock {\em Communications in Partial Differential Equations},
  (32):849--877, 2007.

\bibitem{Wang2014on}
L.~Wang, C.~Mu, and P.~Zheng,
\newblock {On a quasilinear parabolic-elliptic chemotaxis system with logistic
  source},
\newblock {\em J. Differential Equations}, 256:1847--1872, 2014.

%
%
%
\bibitem{Winkler2010aggregation}
M.~Winkler,
\newblock {Aggregation vs. global diffusive behavior in the higher-dimensional
  Keller-Segel model},
\newblock {\em J. Differential Equations}, 248:2889--2905, 2010.

\bibitem{Winkler2011blow}
M.~Winkler,
\newblock {Blow-up in a higher-dimensional chemotaxis system despite logistic
  growth restriction},
\newblock {\em Journal of Mathematical Analysis and Applications},
  384:261--272, 2011.

\bibitem{Winkler2013finite}
M.~Winkler,
\newblock {Finite-time blow-up in the higher-dimensional parabolic-parabolic
  Keller-Segel system},
\newblock {\em J. Math. Pures Appl.}, 100:748--767, 2013.

\bibitem{Winkler2014global}
M.~Winkler,
\newblock {Global asymptotic stability of constant equilibria in a fully
  parabolic chemotaxis system with strong logistic dampening},
\newblock {\em J. Differential Equations}, 257(4):1056--1077, 2014.

\bibitem{Winkler2014how}
M.~Winkler,
\newblock {How far can chemotactic cross-diffusion enforce exceeding carrying
  capacities?}
\newblock {\em J. Nonlinear Sci.}, 24:809--855, 2014.
%
\bibitem{Yokota2015existence}
T.~Yokota and N.~Yoshino,
\newblock {Existence of solutions to chemotaxis dynamics with logistic source},
\newblock {\em Discrete Contin. Dyn. Syst. Dynamical systems, differential
  equations and applications. 10th AIMS Conference. Suppl.}, pages 1125--1133,
  2015.
%
%
%


\bibitem{Zheng2015boundedness}
P.~Zheng, C.~Mu, X.~Hu, and Y.~Tian,
\newblock {Boundedness of solutions in a chemotaxis system with nonlinear
  sensitivity and logistic source},
\newblock {\em . Math. Anal. Appl.}, 424:509--522, 2015.

\end{thebibliography}
\end{document}